\documentclass[10pt]{amsart}

\usepackage[
paper=a4paper,
text={147mm,230mm},centering
]{geometry}
\usepackage{algorithm}
\usepackage{algpseudocode}
\usepackage{float}
\newfloat{algorithm}{t}{lop}
\iftrue
\makeatletter
\def\@settitle{%
  \vspace*{-20pt}
  \begin{flushleft}%
    \baselineskip14\p@\relax
    \normalfont\bfseries\LARGE
    \@title
  \end{flushleft}%
}
\def\@setauthors{%
  \begingroup
  \def\thanks{\protect\thanks@warning}%
  \trivlist
  \large \@topsep30\p@\relax
  \advance\@topsep by -\baselineskip
  \item\relax
  \author@andify\authors
  \def\\{\protect\linebreak}%
  \authors
  \ifx\@empty\contribs
  \else
    ,\penalty-3 \space \@setcontribs
    \@closetoccontribs
  \fi
  \normalfont
  \@setaddresses
  \endtrivlist
  \endgroup
}
\def\@setaddresses{\par
  \nobreak \begingroup\raggedright
  \small
  \def\author##1{\nobreak\addvspace\smallskipamount}%
  \def\\{\unskip, \ignorespaces}%
  \interlinepenalty\@M
  \def\address##1##2{\begingroup
    \par\addvspace\bigskipamount\noindent
    \@ifnotempty{##1}{(\ignorespaces##1\unskip) }%
    {\ignorespaces##2}\par\endgroup}%
  \def\curraddr##1##2{\begingroup
    \@ifnotempty{##2}{\nobreak\noindent\curraddrname
      \@ifnotempty{##1}{, \ignorespaces##1\unskip}\/:\space
      ##2\par}\endgroup}%
  \def\email##1##2{\begingroup
    \@ifnotempty{##2}{\smallskip\nobreak\noindent E-mail address%
      \@ifnotempty{##1}{, \ignorespaces##1\unskip}\/:\space
      \ttfamily##2\par}\endgroup}%
  \def\urladdr##1##2{\begingroup
    \def~{\char`\~}%
    \@ifnotempty{##2}{\nobreak\noindent\urladdrname
      \@ifnotempty{##1}{, \ignorespaces##1\unskip}\/:\space
      \ttfamily##2\par}\endgroup}%
  \addresses
  \endgroup
  \global\let\addresses=\@empty
}

\def\@setabstracta{%
    \ifvoid\abstractbox
  \else
    \skip@25\p@ \advance\skip@-\lastskip
    \advance\skip@-\baselineskip \vskip\skip@
    \box\abstractbox
    \prevdepth\z@ 
    \vskip-10pt
  \fi
}
\renewenvironment{abstract}{%
  \ifx\maketitle\relax
    \ClassWarning{\@classname}{Abstract should precede
      \protect\maketitle\space in AMS document classes; reported}%
  \fi
  \global\setbox\abstractbox=\vtop \bgroup
    \normalfont\small
    \list{}{\labelwidth\z@
      \leftmargin0pc \rightmargin\leftmargin
      \listparindent\normalparindent \itemindent\z@
      \parsep\z@ \@plus\p@
      
    }%
    \item[\hskip\labelsep\bfseries\abstractname.]%
}{%
  \endlist\egroup
  \ifx\@setabstract\relax \@setabstracta \fi
}
\def\section{\@startsection{section}{1}%
  \z@{-1.2\linespacing\@plus-.5\linespacing}{.8\linespacing}%
  {\normalfont\bfseries\large}}
\def\subsection{\@startsection{subsection}{2}%
  \z@{-.8\linespacing\@plus-.3\linespacing}{.3\linespacing\@plus.2\linespacing}%
  {\normalfont\bfseries}}
\def\subsubsection{\@startsection{subsubsection}{3}%
  \z@{.7\linespacing\@plus.1\linespacing}{-1.5ex}%
  {\normalfont\itshape}}
\def\@secnumfont{\bfseries}
\makeatother
\fi 
\usepackage{amssymb}
\usepackage[all]{xy}
\usepackage{graphicx,color,float}
\usepackage[bookmarks]{hyperref}
\usepackage[textwidth=1.3in,color=yellow]{todonotes}
\usepackage{amsmath}
\usepackage{pb-diagram,pb-xy}
\usepackage{url}
\usepackage{fancyvrb}

\usepackage{dcolumn}
\newcolumntype{2}{D{.}{}{2.0}}
\fvset{tabsize=4}

\usepackage{tikz}

\textwidth=\paperwidth \advance\textwidth by-2\oddsidemargin
\advance\oddsidemargin by-1in
\evensidemargin=\oddsidemargin
\calclayout

\def\Z{\mathbb{Z}}

\def\Q{\mathbb{Q}}
\def\R{\mathbb{R}}

\def\SL{\operatorname{SL}(3;\Z)}

\def\+{\oplus}

\def\varepsilon{\epsilon}
\newcommand{\qtheta}[1]{\mathbb{Q}[\Theta_{#1}]}
\newcommand{\ztheta}[1]{\mathbb{Z}[\Theta_{#1}]}

\newcommand{\CS}{Cappell-Shaneson }

\theoremstyle{plain}
\newtheorem{theorem}{Theorem}[section]
\newtheorem{proposition}[theorem]{Proposition}
\newtheorem{corollary}[theorem]{Corollary}
\newtheorem{lemma}[theorem]{Lemma}
\newtheorem{theoremalpha}{Theorem}
\newtheorem{corollaryalpha}[theoremalpha]{Corollary} 
\newtheorem{conjecture}{Conjecture}
\newtheorem*{spc4}{The smooth 4-dimensional Poincar\'{e} conjecture}
\theoremstyle{definition}

\newtheorem*{organization}{Organization of the paper}
\newtheorem{definition}[theorem]{Definition}
\newtheorem{example}[theorem]{Example}
\newtheorem{remark}[theorem]{Remark}

\newtheorem*{remark0}{Remark}
\newtheorem*{acknowledgement}{Acknowledgement}
\newtheorem*{claim}{Claim}
\def\to{\mathchoice{\longrightarrow}{\rightarrow}{\rightarrow}{\rightarrow}}
\makeatletter
\newcommand{\shortxra}[2][]{\ext@arrow 0359\rightarrowfill@{#1}{#2}}
\def\longrightarrowfill@{\arrowfill@\relbar\relbar\longrightarrow}
\newcommand{\longxra}[2][]{\ext@arrow 0359\longrightarrowfill@{#1}{#2}}

\makeatother

\makeatletter
\newcommand*{\@old@slash}{}\let\@old@slash\slash
\def\slash{\relax\ifmmode\delimiter"502F30E\mathopen{}\else\@old@slash\fi}
\makeatother

\newcommand{\Mod}[1]{\ (\mathrm{mod}\ #1)}

\begin{document}

\title [Ideal classes and Cappell-Shaneson homotopy 4-spheres]
{Ideal classes and Cappell-Shaneson homotopy 4-spheres}

\author{Min Hoon Kim}
\address{
  School of Mathematics\\
  Korea Institute for Advanced Study \\
  Seoul 130--722\\
  Republic of Korea
}
\email{kminhoon@kias.re.kr}

\author{Shohei Yamada}
\address{}
\email{fujijyu{\_}alcyone@yahoo.co.jp}

\subjclass[2000]{57R60}


\maketitle
\begin{abstract}Gompf proposed a conjecture on Cappell-Shaneson matrices whose affirmative answer implies that all Cappell-Shaneson homotopy 4-spheres are diffeomorphic to the standard 4-sphere. 
We study Gompf conjecture on Cappell-Shaneson matrices using various algebraic number theoretic techniques. We find a hidden symmetry between trace $n$ Cappell-Shaneson matrices and trace $5-n$ Cappell-Shaneson matrices which was suggested by Gompf experimentally. Using this symmetry, we prove that Gompf conjecture for the trace $n$ case is equivalent to the trace $5-n$ case. We confirm Gompf conjecture for the special cases that $-64\leq \textrm{trace}\leq 69$ and corresponding Cappell-Shaneson homotopy 4-spheres are diffeomorphic to the standard 4-sphere. We also give a new infinite family of Cappell-Shaneson spheres which are diffeomorphic to the standard 4-sphere.
\end{abstract}
 
\section{Introduction}
The smooth 4-dimensional Poincar\'{e} conjecture is a central open problem in low-dimensional topology.
\begin{spc4} Every homotopy $4$-sphere is diffeomorphic to~$S^4$.
\end{spc4}
Cappell and Shaneson \cite{Cappell-Shaneson:1976-1} constructed homotopy 4-spheres, called \emph{Cappell-Shaneson homotopy $4$-spheres}. These homotopy 4-spheres are the most notable, potential counterexamples of the smooth 4-dimensional Poincar\'{e} conjecture. The following folklore conjecture is a special case of the smooth 4-dimensional Poincar\'{e} conjecture and has remained open for 40 years.
\begin{conjecture}\label{conjecture:CS}Every \CS homotopy $4$-sphere is diffeomorphic to $S^4$.
\end{conjecture}

One of our main results, Corollary~\ref{corollary:C}, will give the largest known family of \CS spheres that are diffeomorphic to $S^4$, supporting Conjecture~\ref{conjecture:CS}. To motivate our results, we recall several earlier results on Cappell-Shaneson spheres.
\subsection{Historical background}
Cappell-Shaneson spheres $\Sigma_A^\varepsilon$ are parametrized by  a matrix $A\in \SL$ with $\det(A-I)=1$ and a choice of framing $\varepsilon\in \Z_2$. We say a matrix $A\in \SL$ is a \emph{Cappell-Shaneson matrix} if $\det (A-I)=1$. For example, for any $n\in \Z$, the following matrix $A_n$ is a \CS matrix
\[ A_n=\begin{bmatrix}0&1&0\\0&1&1\\1&0&n+1 \end{bmatrix}.\]

We first recall history on \CS spheres $\Sigma_{A_n}^\epsilon$ corresponding to the family $A_n$ which have been studied thoroughly. For more details, we refer the reader to \cite[Section~14.2]{Akbulut:2016-1} where a nice discussion on $\Sigma_{A_n}^\epsilon$ is given with many handlebody diagrams. Akbulut and Kirby \cite{Akbulut-Kirby:1979-1} proved that $\Sigma_{A_0}^0$ is diffeomorphic to $S^4$ by drawing its handlebody diagram and simplifying the diagram. They claimed that $\Sigma_{A_0}^0$ is the double cover of the \CS fake $\mathbb{R}\mathbb{P}^4$, denoted by $Q$, corresponding to the matrix 
\[\begin{bmatrix}0&1&0\\0&0&1\\-1&1&0\end{bmatrix}\]
which was constructed in  \cite{Cappell-Shaneson:1976-2}. Aitchison and Rubinstein \cite{Aitchison-Rubinstein:1984-1} pointed out that $\Sigma_{A_0}^1$ is indeed the double cover of $Q$. We remark that $Q$ is used by Akbulut to construct several interesting fake non-orientable 4-manifolds in \cite{Akbulut:1982-1,Akbulut:1985-1}, and to show that a Gluck twist can change the diffeomorphism type for a \emph{non-orientable} 4-manifold in  \cite{Akbulut:1988-1}. (It is unknown whether a Gluck twist can change the diffeomorphism type of an \emph{orientable} 4-manifold.)
 In the same paper \cite{Aitchison-Rubinstein:1984-1}, Aitchison and Rubinstein proved that $\Sigma_{A_n}^0$ is diffeomorphic to $S^4$ for all $n$.

For the non-trivial framing case, Akbulut and Kirby \cite{Akbulut-Kirby:1985-1} drew a handlebody diagram of $\Sigma_{A_0}^1$ without 3-handles. They first introduced canceling pairs of 2- and 3-handles to remove 1-handles, and turned the resulting diagram upside-down to obtain the diagram without 3-handles. We remark that similar techniques are used in \cite{Akbulut:1999-1,Akbulut:2002-1,Akbulut:2012-1}. They showed that the punctured $\Sigma_{A_0}^1$ can be embedded in $S^4$. In particular, by topological Sch\"{o}nflies theorem, $\Sigma_{A_0}^1$ is homeomorphic to~$S^4$. (Of course, this fact also can be checked by using Freedman's theorem.) They also observed that $\Sigma_{A_0}^1$ is diffeomorphic to $S^4$ if a balanced presentation of the trivial group 
\[\langle x, y \mid xyx=yxy, x^5=y^4\rangle\]
 is Andrews-Curtis trivial. (This balanced presentation is unlikely Andrew-Curtis trivial.) 

Consequently, if $\Sigma_{A_0}^1$ were not diffeomorphic to $S^4$, then all of the smooth 4-dimensional Poincar\'{e} conjecture, the smooth Sch\"{o}nflies conjecture and the Andrews-Curtis conjecture would be false simultaneously. Gompf \cite{Gompf:1991-1} excluded this possibility by proving that $\Sigma_{A_0}^1$ is actually diffeomorphic to $S^4$ by adding a canceling pair of 2- and 3-handles. After a lengthy handlebody calculus, Gompf \cite{Gompf:1991-2} gave a handlebody diagram of $\Sigma_{A_n}^1$ without 3-handles for each $n$. 

Around three decades later, Freedman, Gompf, Morrison and Walker \cite{Freedman-Gompf-Morrison-Walker:2010-1} tried to disprove the smooth Poincar\'{e} conjecture via the following strategy. They considered knots obtained by adding a band to the two attaching circles of 2-handles in the handlebody diagrams of $\Sigma_{A_n}^1$ given in~\cite{Gompf:1991-2}. A simple, but interesting observation is that such knots are slice in a homotopy 4-ball obtained from $\Sigma_{A_n}^1$ by removing a small open ball. Hence if there is a non-slice knot obtained in this way, then  the smooth 4-dimensional Poincar\'{e} conjecture would be false. By choosing specific bands, they obtained explicit diagrams of such knots, and computed Rasmussen $s$-invariants of them to disprove the smooth 4-dimensional Poincar\'{e} conjecture, but the $s$-invariants of their examples are trivial. 

It turns out that there are underlying reasons that their attempts cannot be successful. Kronheimer and Mrowka \cite{Kronheimer-Mrowka:2013-1} proved that the $s$-invariant vanishes for knots which are slice in a homotopy 4-ball by giving a gauge theoretic interpretation of the $s$-invariant. Akbulut \cite{Akbulut:2010-1} added a marvellous canceling pair of 2- and 3-handles to the handlebody diagram of $\Sigma_{A_n}^1$ given in~\cite{Gompf:1991-2}, and proved that $\Sigma_{A_n}^1$ is diffeomorphic to $S^4$ for any integer $n$.

Now we recall what was known about general \CS spheres. From the construction of \CS spheres, it can be easily seen that two similar \CS matrices give diffeomorphic \CS spheres. More precisely, if $A$ and $B$ are similar \CS matrices, then \CS spheres $\Sigma_{A}^\epsilon$ and $\Sigma_{B}^\epsilon$ are diffeomorphic for any $\epsilon\in \Z_2$. Therefore it is natural to think of the set of the similarity classes of \CS matrices. 

Our starting point is a result of Aitchison and Rubinstein \cite{Aitchison-Rubinstein:1984-1} which states that every \CS matrix is similar to a \emph{standard} \CS matrix
\[X_{c,d,n}=\begin{bmatrix}0&a&b\\0&c&d\\1&0&n-c\end{bmatrix}\]
for some integers $c$, $d$ and $n$ such that $f_n(c)\equiv 0\Mod{d}$ where $f_n(x)=x^3-nx^2+(n-1)x-1$. Note that $f_n(x)$ is the minimal polynomial of $X_{c,d,n}$ and the entries $a$ and $b$ are determined as $b=(c-1)(n-c-1)$ and $ad=f_n(c)$ from the equalities $\det(X_{c,d,n})=\det(X_{c,d,n}-I)=1$.

Moreover, using a classical result of Latimer-MacDuffee and Taussky \cite{Latimer-MacDuffee:1933-1,Taussky:1949-1}, Aitchison and Rubinstein \cite{Aitchison-Rubinstein:1984-1} observed that for any integer $n$, there are only \emph{finitely} many similarity classes of trace $n$ \CS matrices. In fact, there is a bijection between the set of similarity classes of trace $n$ \CS matrices and the ideal class monoid $C(\Z[\Theta_n])$ where $\Theta_n$ is a root of $f_n(x)$. Via the bijection, the similarity class of $X_{c,d,n}$ corresponds to the ideal class $[\langle \Theta_n-c,d\rangle]$. (In particular, $A_n=X_{1,1,n+2}$ corresponds to the identity element in $C(\Z[\Theta_{n+2}])$, see Remark~\ref{remark:identity}.) 

In short, to confirm Conjecture~\ref{conjecture:CS}, it suffices to show that $\Sigma_{X_{c,d,n}}^\epsilon$ is diffeomorphic to $S^4$ for finitely many pairs of such integers $c$, $d$ for each integer $n$. (It is sufficient to check this for each representative of $C(\Z[\Theta_n])$.) However, this has been untouched mainly because finding and simplifying their handlebody diagrams of $\Sigma_{X_{c,d,n}}^\epsilon$ seem to be an onerous task.

 In \cite{Gompf:2010-1}, Gompf proved that \CS spheres $\Sigma^\epsilon_{X_{c,d,n}}$ and $\Sigma^{\epsilon}_{X_{c,d,n+kd}}$ are diffeomorphic for any $\epsilon\in \Z_2$ and any integer $k$. It is remarkable that Gompf's proof does not involve any handlebody diagram.  Nonetheless, this method is strong enough to give an alternative proof of the aforementioned result of Akbulut that $\Sigma_{A_n}^1$ is actually diffeomorphic to $S^4$ for any integer $n$. (To see this, note that $A_n=X_{1,1,n+2}$, and hence $\Sigma_{A_n}^\epsilon$ is diffeomorphic to  $\Sigma_{A_0}^\epsilon$ which is diffeomorphic to $S^4$ by \cite{Akbulut-Kirby:1979-1,Gompf:1991-1}.) Using a computation of $C(\Z[\Theta_{-5}])$ by Aitchison and Rubinstein, Gompf showed that more \CS spheres are standard. Indeed, two \CS spheres $\Sigma_{X_{2,3,-5}}^0$ and $\Sigma_{X_{2,3,-5}}^1$ correspond to the \CS matrix
\[X_{2,3,-5}=\begin{bmatrix}0&-5&-8\\0&2&3\\1&0&-7\end{bmatrix}\]
are diffeomorphic to $S^4$. This result does not follow from the result of Akbulut since $X_{2,3,-5}$ is not similar to $A_{-5}$.  
 
In this context, Gompf considered an equivalence relation on the set of standard \CS matrices generated by similarity and $X_{c,d,n}\sim_G X_{c,d,n+kd}$ for $k\in \Z$. (We will call the equivalence relation by \emph{Gompf equivalence}.) By the aforementioned result of Gompf, if two \CS matrices are Gompf equivalent, then they give diffeomorphic \CS spheres. Gompf conjectured the following  whose affirmative answer implies Conjecture~\ref{conjecture:CS}. 
\begin{conjecture}[{\cite[Conjecture 3.6]{Gompf:2010-1}}] \label{conjecture:Gompf}Every \CS matrix is Gompf equivalent to~$A_0$. 
\end{conjecture}

 \subsection{Main results}

In this paper, using various techniques in algebraic number theory, we study Conjecture~\ref{conjecture:Gompf} in a systematic way. For brevity of our discussion, we say \emph{Conjecture~\ref{conjecture:CS} is true for a \CS matrix A} if $\Sigma_{A}^\epsilon$ is diffeomorphic to $S^4$ for every $\epsilon\in \Z_2$. Similarly, we say  \emph{Conjecture~\ref{conjecture:Gompf} is true for trace $n$} if every \CS matrix $A$ with trace $n$ is Gompf equivalent to $A_0$. 

\begin{remark}\label{remark:Gompfconjectureimplications} If Conjecture~\ref{conjecture:Gompf} is true for trace $n$, then Conjecture~\ref{conjecture:CS} is true for every \CS matrix with trace $n$. More generally, $\Sigma_{X_{c,d,n+kd}}^\epsilon$ is diffeomorphic to $S^4$ for any $k\in \Z$ and $\epsilon\in\Z_2$.
\end{remark}

Our first result, Theorem~\ref{theorem:A}, shows that there is a hidden symmetry between trace~$n$ \CS matrices and  trace $5-n$ \CS matrices. Theorem~\ref{theorem:A} implies that if Conjecture~\ref{conjecture:Gompf} is true for trace $n\geq 3$, then both of Conjectures~\ref{conjecture:CS} and~\ref{conjecture:Gompf} will be true simultaneously.
 \begin{theoremalpha}\label{theorem:A}There is a bijection between the set of similarity classes of trace $n$ \CS matrices and the set of similarity classes of trace $5-n$ \CS matrices. Moreover, Conjecture~\ref{conjecture:Gompf} is true for trace $n$ if and only if Conjecture~\ref{conjecture:Gompf} is true for trace $5-n$ for any integer $n$. 
 \end{theoremalpha}
 
  To prove Theorem~\ref{theorem:A}, we will explicitly give a ring isomorphism from $\Z[\Theta_n]$ to $\Z[\Theta_{5-n}]$ which induces a monoid isomorphism between $C(\Z[\Theta_n])$ and $C(\Z[\Theta_{5-n}])$. This gives a bijection between the set of similarity classes of trace $n$ \CS matrices and the set of similarity classes of trace $5-n$ \CS matrices. We will observe that the bijection is compatible with Gompf equivalence, and Theorem~\ref{theorem:A} will follow from the observation. 
    
Our second result, Theorem~\ref{theorem:B}, shows that Conjecture~\ref{conjecture:Gompf} is true for trace $n$ if $|n|$ is small. 

 \begin{theoremalpha}\label{theorem:B}Conjecture~\ref{conjecture:Gompf} is true for trace $n$ if $-64\leq n\leq 69$. \end{theoremalpha}

To prove Theorem~\ref{theorem:B}, we first find representatives of elements of $C(\Z[\Theta_n])$. (Equivalently, we find a representative for each similarity class of trace $n$ \CS matrices.) When $\Z[\Theta_n]$ is a Dedekind domain, this task can be done using MAGMA software (see Section~\ref{section:findingrepresentatives}).

When $\Z[\Theta_n]$ is not a Dedekind domain, the current version of MAGMA cannot compute $C(\Z[\Theta_n])$. (Nonetheless, using MAGMA, we can still compute a strictly smaller subset $\operatorname{Pic}(\Z[\Theta_n])$ of $C(\Z[\Theta_n])$, consisting of the classes of invertible ideals.) We will observe that there are infinitely many integers $n$ such that $\Z[\Theta_n]$ is not a Dedekind domain. In fact, for each integer~$k$, $\Z[\Theta_{49k+27}]$ is not a Dedekind domain (see Proposition~\ref{proposition:49k+27}). Consequently, when we prove Theorem~\ref{theorem:B}, it is the most difficult to confirm that Conjecture~\ref{conjecture:Gompf} is true for trace $27$. Using Dedekind-Kummer theorem, we analyze non-invertible ideals of $\Z[\Theta_{27}]$ explicitly (for details, see Section~\ref{subsection:27}), and determine the monoid structure of $C(\Z[\Theta_{27}])$. The authors think that our method could also be used to study $C(\Z[\Theta_n])$ for general~$n$ such that $\Z[\Theta_n]$ is not a Dedekind domain.

By Theorem~\ref{theorem:A}, to prove Theorem~\ref{theorem:B}, it suffices to confirm that Conjecture~\ref{conjecture:Gompf} is true for trace $3\leq n\leq 69$. In Tables~\ref{table:3<=n<=36}--\ref{table:61<=n<=69}, we give representatives of elements of $C(\Z[\Theta_n])$ for $3\leq n\leq 69$. We have to show that the corresponding standard \CS matrices are Gompf equivalent to $A_0$. Recall that Gompf equivalence is an equivalence relation on the set of standard \CS matrices generated by similarity and $X_{c,d,n}\sim_G X_{c,d,n+kd}$ for $k\in \Z$. Understanding when two \emph{standard} \CS matrices are similar is important to study Conjecture~\ref{conjecture:Gompf}, but this seems to be a difficult question in algebraic number theory. Instead, for any given standard \CS matrix, we give a MAGMA code which gives a list of \CS matrices with sufficiently small entries in Section~\ref{section:findingrepresentatives}. Using this,  we could find several non-trivial Gompf equivalences. The authors think that finding such Gompf equivalences by hands is cumbersome. 

In \cite[Theorem~3.1]{Earle:2014-1}, Earle considered the following special family of \CS matrices
\[X_{c,d,c+2}=\begin{bmatrix}0&a&b\\0&c&d\\1&0&2\end{bmatrix},\]
and showed that $X_{c,d,c+2}$ are Gompf equivalent to $A_0$ if $0\leq c\leq 94$ and $a\neq 19,37$, or if $1\leq d\leq 35$. Earle found similar \CS matrices by hands. As an application of our method, using our MAGMA codes, we recover and generalize the result of Earle. Indeed, we show that the \CS matrices $X_{c,d,c+2}$ are Gompf equivalent to $A_0$ if $0\leq c\leq 94$, or if $1\leq d\leq 134$ by removing technical conditions on the entry $a$, and weakening the condition on the entry $d$ (see Theorem~\ref{theorem:Earlegeneralversion}).
  
Theorem~\ref{theorem:B} enables us to find new \CS spheres that are diffeomorphic to~$S^4$, which we record the result as Corollary~\ref{corollary:C}. By Remark~\ref{remark:Gompfconjectureimplications}, Corollary~\ref{corollary:C} immediately follows from Theorem~\ref{theorem:B}. 
 \begin{corollaryalpha}\label{corollary:C}Conjecture~\ref{conjecture:CS} is true for trace $n$ \CS matrices if $n$ is an integer such that $-64\leq n\leq 69$. More generally, $\Sigma_{X_{c,d,n}}^\epsilon$ is diffeomorphic to $S^4$ for any $\epsilon\in \Z_2$ and for any integers $c$, $d$ and $n$ that satisfy $f_n(c)\equiv 0\Mod d$ and $n\equiv n_0\Mod d$ for some $-64\leq n_0\leq 69$. In particular, $\Sigma_{X_{c,d,n}}^\epsilon$ is diffeomorphic to $S^4$ for any $\epsilon\in \Z_2$ if $|d|\leq 134$.
 \end{corollaryalpha}
 By Corollary~\ref{corollary:C}, to find a counterexample to Conjecture~\ref{conjecture:CS}, one should start from a \CS matrix whose trace is either greater than~$69$ or less than~$-64$. We remark that Corollary~\ref{corollary:C} gives the largest known family of \CS spheres which are diffeomorphic to~$S^4$.

\begin{remark}In Tables~\ref{table:3<=n<=36}--\ref{table:61<=n<=69}, we give the lists of representatives $(c,d,n)$ of elements of $C(\Z[\Theta_n])$ for $3\leq n\leq 69$. Each tuple $(c,d,n)$ corresponds to the standard \CS matrix $X_{c,d,n}$. For example, when $n=21$, there are three corresponding tuples $(1,1,21)$, $(5,7,21)$ and $(9,13,21)$ in Table~\ref{table:3<=n<=36}. This means that every \CS matrix $A$ with $\operatorname{tr}(A)=21$ is similar to exactly one of the following three matrices: 
\[X_{1,1,21}=\begin{bmatrix}0&1&0\\0&1&1\\ 1&0&20\end{bmatrix},\, X_{5,7,21}=\begin{bmatrix}0&43&60\\0&5&7\\1&0&16\end{bmatrix},\, X_{9,13,21}=\begin{bmatrix}0&61&88\\0&9&13\\1&0&12\end{bmatrix}. \]
Note that $X_{1,1,21}=A_{19}$. Using this computation and Theorem~\ref{theorem:A}, we can see that there are 1314 non-trivial ideal classes of $C(\Z[\Theta_n])$ for $-64\leq n\leq 69$. In particular, Corollary~\ref{corollary:C} gives at least 2628 \CS spheres that are diffeomorphic to $S^4$, and this fact is not covered by the result of Akbulut \cite{Akbulut:2010-1}.
\end{remark}

   It is natural to ask whether Corollary~\ref{corollary:C} actually gives a new infinite family of \CS matrices whose corresponding \CS spheres are diffeomorphic to $S^4$. Our final result, Corollary~\ref{corollary:D}, shows that this is the case. For this purpose, we consider the following family of \CS matrices $M_k$ ($k\in \Z$),
	\[ M_{k}=\begin{bmatrix}0&14k+7&49k+24\\0&2&7\\1&0&49k+25
	\end{bmatrix}.\] 
	Note that $M_k=X_{2,7,49k+27}$, and hence $\Sigma_{M_k}^\epsilon$ is diffeomorphic to $S^4$ for any $\epsilon\in \Z_2$ by Corollary~\ref{corollary:C}. (This fact can be also checked by using a weaker version given in \cite[Theorem~3.2]{Gompf:2010-1}.) We show that $M_k$ is not similar $A_n$ for any integers $k$ and $n$. 
 \begin{corollaryalpha}\label{corollary:D}For any integers $k$ and $\epsilon\in \Z_2$, \CS sphere $\Sigma_{M_k}^\epsilon$ corresponding to $M_k$ is diffeomorphic to $S^4$. For any integers $k$ and $n$, $M_k$ is not similar to $A_n$.
 \end{corollaryalpha}
  Recall that Akbulut \cite{Akbulut:2010-1} showed that the infinite family of \CS matrices $A_n$ give \CS spheres $\Sigma_{A_n}^\epsilon$ are diffeomorphic to $S^4$ for any $\epsilon\in \Z_2$. Since $M_k$ is not similar to $A_n$ for any integers $k$ and $n$, Corollary~\ref{corollary:D} is not covered by the result of Akbulut.

 \begin{organization} In Section~\ref{section:preliminaries}, we recall several facts on \CS spheres and \CS matrices, and we discuss the correspondence between ideal class monoid and the similarity classes of \CS matrices. In Section~\ref{section:symmetry}, we prove Theorem~\ref{theorem:A}. In Section~\ref{section:idealclassmonoid}, we recall Dedekind-Kummer theorem, and show that $C(\Z[\Theta_{49k+27}])$ is not a group for any integer $k$, and discuss the structure of $C(\Z[\Theta_{27}])$. In Section~\ref{section:findingrepresentatives}, we use MAGMA software to find representatives of elements in $\operatorname{Pic}(\Z[\Theta_{n}])$. In Section~\ref{section:theoremB}, we prove Theorem~\ref{theorem:B} and Corollary~\ref{corollary:D}. In Section~\ref{section:Earle}, we give a generalization of the result of Earle.
 \end{organization}
 \begin{acknowledgement}The authors would like to thank Tetsuya Abe, Selman Akbulut, Jae Choon Cha, Hisaaki Endo, Robert Gompf, Mark Powell and Motoo Tange for their encouragements and helpful discussions. The first author would like to thank Jung Won Lee for helping him to use MAGMA software. The first author was partially supported by the POSCO TJ Park Science Fellowship.
 \end{acknowledgement}
 
\section{Preliminaries}\label{section:preliminaries}
In this section, we collect several facts on \CS spheres and matrices following \cite[Appendix]{Aitchison-Rubinstein:1984-1} and \cite{Gompf:2010-1}. 
\subsection{\CS spheres and matrices}\label{subsection:csmatrices}
Let $\SL$ be the set of $3\times 3$ integral matrices whose determinants are $1$. We say two matrices $A, B\in \SL$ are \emph{similar} if there is a matrix $C\in \SL$ such that $A=CBC^{-1}$. 
\begin{definition}
 A matrix $A\in \SL$ is a \textit{\CS matrix} if $A-I\in \SL$.
\end{definition}

	For a \CS matrix $A\in \SL$, Cappell and Shaneson \cite{Cappell-Shaneson:1976-1} constructed two homotopy 4-spheres $\Sigma_A^\varepsilon$ as follows. Let $T^3$ be the 3-torus $\R^3/\Z^3$. Since $A\in \SL$, $A$ induces an orientation-preserving diffeomorphism $f_A\colon T^3\to T^3$. Possibly after an isotopy, we can assume that $f_A$ is the identity on a neighborhood $D_y$ of some chosen point $y\in T^3$. Let $W_A$ be the mapping torus of $f_A$, that is,
	\[W_A=T^3\times [0,1]/(x,0)\sim(f_A(x),1).\]
 Since $f_A$ is the identity around the point $y$, we can regard $D_y\times S^1\subset W_A$. From the condition $\det (A-I)=1$, the Wang sequence applied to the fiber bundle $T^3\hookrightarrow W_A\to S^1$, and Van Kampen theorem show that $W_A$ is a homology $S^1\times S^3$ whose fundamental group $\pi_1(W_A)$ is normally generated by $[y\times S^1]$. If we remove $D_y\times S^1$ from $W_A$ and glue $S^2\times D^2$ along the boundary via a framing $\epsilon \in \Z_2$, then we obtain a homotopy 4-sphere $\Sigma_A^\epsilon$.
 
\begin{definition}[\CS spheres] For a \CS matrix $A$, two homotopy 4-spheres $\Sigma_A^0$ and $\Sigma_A^1$ are called \emph{\CS spheres} correspond to $A$.
\end{definition}
 \begin{remark}\label{remark:similarimpliesdiffeomorphic}From the construction of \CS homotopy 4-spheres, if $A$ and $B$ are similar \CS matrices, then $W_A$ and $W_B$ are diffeomorphic, and hence $\Sigma_A^\varepsilon$ and $\Sigma_B^\varepsilon$ are diffeomorphic.
 \end{remark}
 By Remark~\ref{remark:similarimpliesdiffeomorphic}, to study \CS spheres up to diffeomorphism, it is natural to consider the similarity classes of \CS matrices. In \cite[Appendix]{Aitchison-Rubinstein:1984-1}, the similarity classes of \CS matrices in terms of ideal classes are systematically studied using a result of Latimer-MacDuffee and Taussky \cite{Latimer-MacDuffee:1933-1,Taussky:1949-1} which we recall in below.
 
  Let $A$ be a \CS matrix with trace $n$. The characteristic polynomial of $A$ is 
	\[f_n(x) =x^3 -n x^2 +(n-1)x -1.\]
\begin{remark} For $A\in \SL$, $A$ is a \CS matrix with trace $n$ if and only if the characteristic polynomial of $A$ is $f_n(x)$.
	Note that $f_n(x)$ is irreducible over $\Z$ for all $n$ (for example, see \cite[Lemma A4]{Aitchison-Rubinstein:1984-1}).
\end{remark}
\begin{definition}[{\cite{Aitchison-Rubinstein:1984-1, Gompf:2010-1}}]We say a \CS matrix is called \emph{standard} if it is of the form 
\[X_{c,d,n}=
		\begin{bmatrix}
		0 & a & b \\
		0 & c & d \\
		1 & 0& n-c\\
		\end{bmatrix}.
		\]
\end{definition}
By the following theorem of Aitchison and Rubinstein and Remark~\ref{remark:similarimpliesdiffeomorphic}, we restrict our attention to \CS spheres $\Sigma_A^\varepsilon$ which correspond to standard \CS matrices $A$ since we are interested in their differentiable structures.
\begin{theorem}[Aitchison and Rubinstein \cite{Aitchison-Rubinstein:1984-1}]\label{csm} Every Cappell-Shaneson matrix is similar to a standard Cappell-Shaneson matrix. 
\end{theorem}
\begin{remark}Technically, Aitchison and Rubinstein \cite{Aitchison-Rubinstein:1984-1} proved that every \CS matrix is similar to the \emph{transpose} of a standard \CS matrix. However, this is clearly an equivalent statement.
\end{remark}

\begin{remark}\label{remark:CS}Since $X_{c,d,n}$ is a \CS matrix, $\det (X_{c,d,n}-I)=1$ and $\det (X_{c,d,n})=1$. From these conditions, $b=(c-1)(n-c-1)$ and $ad-bc=1$, that is, $X_{c,d,n}$ is uniquely determined by $c,d$ and $n$.
\end{remark}
\begin{remark}\label{remark:Gompfnotation}Gompf \cite{Gompf:2010-1} considered slightly general matrices of the form
\[A=\begin{bmatrix}
		0 & a & b \\
		0 & c & d \\
		1 & e& n-c\\
		\end{bmatrix}
		\]
which Gompf called $A$ is a \emph{\CS matrix in the standard form}. Gompf proved that $\Delta^kA$ and $A\Delta^k$ are Cappell-Shanseon matrices and the corresponding Cappell-Shaneson homotopy spheres $\Sigma_{\Delta^k A}^\varepsilon$ and $\Sigma_{A\Delta^k}^\varepsilon$ are diffeomorphic to $\Sigma_{A}^\varepsilon$ for $\varepsilon=0,1$ where 
	\[
		\Delta =
		\begin{bmatrix}
		1 & -1 & 0  \\
		0 & 1  & 0 \\
		0 & 1  & 1 \\
		\end{bmatrix}.
		\]
Observe that $A$ is equivalent to $X_{c,d,n}$ as follows. (Note that the values of $c$, $d$ and $n$ are preserved.)
\[\begin{bmatrix}
1 & e & 0 \\
0 & 1 & 0 \\
0 & 0 & 1 \\
\end{bmatrix}
\begin{bmatrix}
0 & a & b \\
0 & c & d \\
1 & e & n-c \\
\end{bmatrix}\begin{bmatrix}
1 & -e & 0 \\
0 & 1  & 0 \\
0 & 0  & 1 \\
\end{bmatrix}
=
\begin{bmatrix}
0 & a+ce & b+de \\
0 & c & d \\
1 & 0 & n-c \\
\end{bmatrix}.
\]
We observe that both $\Delta^kA$ and $A\Delta^k$ are similar to $X_{c,d,n+kd}$ as follows. Note that $\Delta^kA$ and $A\Delta^k$ are similar because $\Delta^kA=\Delta^k(A\Delta^k)\Delta^{-k}$. The above argument shows that the matrix 
\[
\Delta^{k} A=
	\begin{bmatrix}
	0 & a-kc & b-kd \\
	0 & c & d \\
	1 & kc + e & kd + n-c\\
	\end{bmatrix}
\]
is similar to $X_{c,d,n+kd}$.
\end{remark}

	We end this subsection by giving a simple, algebraic characterization of standard Cappell-Shaneson matrices which will be used frequently.

	\begin{proposition}\label{csm2}
		For integers $c,d\neq 0$ and $n$, the following are equivalent.
		\begin{enumerate}
			\item $f_n (c) \equiv 0 \Mod{d}$.
			
			\item There exist integers $a$ and $b$ such that
			\[X_{c,d,n}=
			\begin{bmatrix}
			0 & a & b \\
			0 & c & d \\
			1 & 0 & n-c \\
			\end{bmatrix}
			\]
			is a standard \CS matrix.
		\end{enumerate}
	\end{proposition}
	
	\begin{proof}We first note that $f_n(c)=c^3-nc^2+(n-1)c-1=-c(c-1)(n-c-1)-1$. Suppose that $f_n (c) \equiv 0 \Mod{d}$. Define $a,b \in \mathbb{Z}$ by
		 $f_n (c) = -ad$ and $b=(c-1)(n-c-1)$. Consider the following matrix
				\[
		A=
		\begin{bmatrix}
		0 & a & b \\
		0 & c & d \\
		1 & 0 & n-c \\
		\end{bmatrix}.
		\]
		
		Note that $A$ is a \CS matrix (and hence is equal to $X_{c,d,n}$) since 
		\[\det A=ad-bc=-f_n(c)-c(c-1)(n-c-1)=1\]
		 and 
		\[\det (A-I)= -(c-1)(n-c-1) +(ad -b(c-1))=1.\]
		
		For the converse, consider the \CS matrix		\[
		X_{c,d,n}=
		\begin{bmatrix}
		0 & a & b \\
		0 & c & d \\
		1 & 0 & n-c \\
		\end{bmatrix}.
		\]
		By Remark~\ref{remark:CS}, $f_n(c)=-c(c-1)(n-c-1)-1=-bc-1=-ad\equiv 0\Mod{d}$. 	\end{proof}	
		\subsection{Latimer-MacDuffee-Taussky correspondence}\label{subsection:Latimer-MacDuffee-Taussky}
	In this subsection, we recall a classical result due to  Latimer-MacDuffee and Taussky \cite{Latimer-MacDuffee:1933-1,Taussky:1949-1}. For more details, see Newman's book \cite{Newman:1972-1}.

	Let $R$ be an integral domain and $\mathcal{I}(R)$ be the set of nonzero ideals of $R$.
	Define an equivalence relation $\approx$ on $\mathcal{I}(R)$ by $I \approx J$ if and only if  there exist non-zero elements $\alpha ,\beta$
	such that $\alpha I = \beta J$.
	Each equivalence class is called \emph{an ideal class} and the ideal class of $I \in \mathcal{I}(R)$ is
	denoted by $[I]$.
	The set of all ideal classes is called the \emph{ideal class monoid} of $R$ denoted by $C(R)$. The multiplication is given by the multiplication of ideals: 
	$[I] \cdot [J] = [IJ]$.
	The identity element is the class of principal ideals. An ideal $I$ of $R$ is called \emph{invertible} if there exists an ideal  $J$ of $R$ such that $IJ$ is a principal ideal. The subset of $C(R)$ which consists of the ideal classes of invertible ideals of $R$ is an abelian group, called the \emph{Picard group} of $R$ and denoted by $\operatorname{Pic}(R)$.
	
	\begin{remark}
We remark that the monoid $C(R)$ is not a group in general. In fact, the following are equivalent for an integral domain $R$:
\begin{enumerate}
\item $R$ is a Dedekind domain.
\item Every ideal of $R$ is invertible.
\item $C(R)$ is a group.
\item $C(R)=\operatorname{Pic}(R)$.
\end{enumerate}
\end{remark}

\begin{example}If $R$ is the ring of integers of an algebraic number field, then $R$ is a Dedekind domain and hence $C(R)$ is a group. 
\end{example}
We are mainly interested in the special case that $R=\Z[\Theta]$ where $\Theta$ is a root of a monic polynomial $g(x) \in \mathbb{Z}[x]$ which is irreducible (over $\Z$). Note that $\Q[\Theta]$ is the number field obtained by adjoining $\Theta$ to $\Q$.

	We recall a classical result due to Latimer-MacDuffee \cite{Latimer-MacDuffee:1933-1} and Taussky \cite{Taussky:1949-1}. For simplicity and our purposes, we spell out the degree $3$ case only. For more details and generalizations, we refer the reader to \cite{Newman:1972-1}.
	\begin{theorem}[Latimer-MacDuffee \cite{Latimer-MacDuffee:1933-1}, Taussky \cite{Taussky:1949-1}]\label{theorem:LMT}
		Suppose $g \in \mathbb{Z}[x]$ is a monic, irreducible polynomial of degree $3$. Let  $\Theta$ be a root of $g$.
		Then there is a bijection between $C(\ztheta{})$ and the set of similarity classes of matrices whose characteristic polynomials are $g$. 
	\end{theorem}
	
	We describe an explicit description of the bijection. Let $A$ be a $3\times 3$ matrix whose characteristic polynomial is $g$ and let $K=\qtheta{}$.
	Regard $A$ as a $K$-linear map $A\colon K^3 \to K^3$.	 Then $\Theta$ is an eigenvalue of $A$ and there exists a corresponding eigenvector in $K^3$.
	In addition, the eigenvalues of $A$ are distinct, 
	because $g$ is irreducible over $\mathbb{Q}$.
	It follows that any two eigenvectors of $A$ corresponding to $\Theta$ are proportional.
	Let  $x = (x_1,x_2, x_3)$ be an eigenvector of $A$ corresponding to $\Theta$.
	We may assume that each $x_i$ lies in $\mathbb{Z}[\Theta]$ by multiplying some integer. Let $I$ be the $\mathbb{Z}$-module generated by $x_1,x_2$ and $x_3$. Then, $I$ is an ideal of $\ztheta{}$. The ideal class $[I]\in C(\Z[\Theta])$ is independent of the choice of an eigenvector $(x_1,x_2,x_3)$, and called \emph{the ideal class which corresponds to $A$}.

	\subsection{The ideal class which corresponds to a Cappell-Shaneson matrix}
	Aitchison and Rubinstein \cite{Aitchison-Rubinstein:1984-1} applied  Theorem~\ref{theorem:LMT} to Cappell-Shaneson matrices which we recall in below for the reader's convenience. Let $\Theta _n$ be a root of $f_n(x)=x^3-nx^2+(n-1)x-1$.	Recall that the set of Cappell-Shaneson matrices with trace $n$ is exactly the set of $3\times 3$ integral matrices $A$ whose characteristic polynomial is $f_n(x)$. Since $f_n(x)$ is irreducible, Theorem~\ref{theorem:LMT} gives a bijection between the set of similarity classes of Cappell-Shaneson matrices with trace $n$
	and $C(\ztheta{n})$. We will explicitly describe the bijection.

	Consider a Cappell-Shaneson matrix with trace $n$,
	\[
	X_{c,d,n}=\begin{bmatrix}
	0 & a & b \\
	0 & c & d \\
	1 & 0 & n-c \\
	\end{bmatrix}.\]
	We find an eigenvector $x=(x_1, x_2, x_3)\in \ztheta{n}^3$ of $X_{c,d,n}$ corresponding to $\Theta _n$.	\[
	\begin{bmatrix}
	0 & a & b \\
	0 & c & d \\
	1 & 0 & n-c \\
	\end{bmatrix}
	\begin{bmatrix}
	x_1 \\
	x_2 \\
	x_3 \\
	\end{bmatrix}
	=
	\begin{bmatrix}
	\Theta_n x_1 \\
	\Theta_n x_2 \\
	\Theta_n x_3 \\
	\end{bmatrix}
	,\quad\quad
	\begin{cases}
	a x_2 +b x_3 =\Theta _n x_1, \\
	c x_2 +d x_3 =\Theta _n x_2, \\
	x_1+(n-c) x_3 =\Theta_n x_3. \\
	\end{cases}
	\]
	In particular, $(x_1,x_2,x_3)=((\Theta_n-n+c)(\Theta_n-c),d,\Theta_n-c)$ is  an eigenvector of $A$ in $\Z[\Theta_n]^3$ with the eigenvalue~$\Theta_n$.	Note that $\langle (\Theta_n-n+c)(\Theta_n-c),d,\Theta_n-c\rangle=\langle \Theta_n-c,d\rangle$. Hence the ideal class $[\langle\Theta_n-c,d\rangle]$ corresponds to the standard \CS matrix $X_{c,d,n}$ by Theorem~\ref{theorem:LMT}. 
	\begin{proposition}[{\cite[page 44]{Aitchison-Rubinstein:1984-1}}]\label{proposition:Aitchison-Rubinstein}
		There is a one-to-one correspondence between the set of similarity classes of 
		Cappell-Shaneson matrices with trace $n$ and $C(\ztheta{n} )$, which is defined by
		\[X_{c,d,n}=
		\begin{bmatrix}
		0 & a & b \\
		0 & c & d \\
		1 & 0 & n-c \\
		\end{bmatrix}
		\mapsto 
		[\langle \Theta_n -c,d \rangle ]
		\]
		where $f_n(c)\equiv0\Mod{d}$, $b=(c-1)(n-c-1)$ and $ad-bc=1$.
	\end{proposition}
	
	\begin{remark}\label{remark:changenbyc}For $k\in \Z$, by Proposition~\ref{proposition:Aitchison-Rubinstein}, $X_{c,d,n}$ and $X_{c+kd,d,n}$ are similar because $\langle\Theta_n-c,d\rangle=\langle\Theta_n-c-kd,d\rangle$.
	\end{remark}

	\subsection{Gompf equivalences and a reformulation of Gompf conjecture}\label{subsection:Gompfequivalence}
	In \cite{Gompf:2010-1}, Gompf introduced a certain equivalence relation (which we call \emph{Gompf equivalences}) between standard \CS matrices which preserve the diffeomorphism types of the corresponding \CS homotopy 4-spheres. We recall Gompf equivalences and give a reformulation of Conjecture~\ref{conjecture:Gompf} in Conjecture~\ref{conjecture:reformulation}.
		
	As in Remark~\ref{remark:Gompfnotation}, let $\Delta$ be the following matrix,
		\[
		\Delta =
		\begin{bmatrix}
		1 & -1 & 0  \\
		0 & 1  & 0 \\
		0 & 1  & 1 \\
		\end{bmatrix}.
		\]
	\begin{theorem}[{\cite[page 1673]{Gompf:2010-1}}]\label{theorem:Gompf}
		
		Let $A$ be a \CS matrix given by 
		\[\begin{bmatrix}0&a&b\\0&c&d\\ 1&e&n-c\end{bmatrix}.\]
		Then, $A \Delta^k$ and $\Delta^kA$ are also \CS matrices
		and corresponding \CS spheres $\Sigma ^\varepsilon _{A \Delta^k}$ and  $\Sigma^\varepsilon_{\Delta^kA}$ are diffeomorphic to $\Sigma ^\varepsilon _{A}$ 
		for every integer $k$ and $\epsilon\in \Z_2$.
	\end{theorem}
	\begin{remark}In Remark~\ref{remark:Gompfnotation}, we remarked that if $A$ is a \CS matrix given by 
		\[A=\begin{bmatrix}0&a&b\\0&c&d\\ 1&e&n-c\end{bmatrix},\]
		then $A$ is similar to $X_{c,d,n}$ and $\Delta^kA$ and $A\Delta^k$ are similar to $X_{c,d,n+kd}$. We know that similar \CS matrices give diffeomorphic homotopy 4-spheres by Remark~\ref{remark:similarimpliesdiffeomorphic}. Therefore, the content of Theorem~\ref{theorem:Gompf} is that two standard \CS matrices $X_{c,d,n}$ and $X_{c,d,n+kd}$ give diffeomorphic homotopy 4-spheres for any integer $k$.
	
	\end{remark}
\begin{definition}[Gompf equivalence] Define an equivalence relation $\sim$, called \emph{Gompf equivalence}, on the set of standard \CS matrices generated by $\sim_S$ and $\sim_G$ where
\begin{alignat*}{2}&X_{c_0,d_0,n}\sim_S  X_{c_1,d_1,n}\quad&&\textrm{~if~}[\langle \Theta_n-c_0,d_0\rangle]=[\langle \Theta_n-c_1,d_1\rangle]\in C(\Z[\Theta_n]),\\
&X_{c,d,n}\sim_GX_{c,d, n+kd}\quad&&\textrm{~if~} k\in \Z.
\end{alignat*}

\end{definition}
	
	Using Theorem~\ref{theorem:Gompf} and Aitchison-Rubinstein's computation of $C(\Z[\Theta_n])$ for small $n$, Gompf proved that Conjecture~\ref{conjecture:Gompf} is true for trace $n$ if $-6\leq n\leq 9$ or $n=11$. In Section~\ref{subsection:proofofTheoremB}, we will show that Conjecture~\ref{conjecture:Gompf} is true for trace $n$ if $-64\leq n \leq 69$.

	\begin{theorem}[{\cite[Theorem 3.2]{Gompf:2010-1}}]\label{eval9}Conjecture~\ref{conjecture:Gompf} is true for trace $n$ if $-6\leq n\leq 9$ or $n=11$.
	\end{theorem}

	We end this preliminary section by giving a reformulation of Conjecture~\ref{conjecture:Gompf}. This reformulation will be convenient to give the proof of Theorem~\ref{theorem:B} given in Section~\ref{section:theoremB}. Let $\Theta_n$ be a root of a polynomial $f_n(x)=x^3-nx^2+(n-1)x-1$. Consider 
	\[\mathcal{C}\mathcal{S}=\{(c,d,n)\in \Z^3\mid f_n(c)\equiv0\Mod{d}\textrm{ and }d\neq 0\}.\]
By Proposition~\ref{csm2}, there is a bijection between $\mathcal{C}\mathcal{S}$ and the set of standard \CS matrices such that the tuple $(c,d,n)\in \mathcal{C}\mathcal{S}$ corresponds to the standard \CS matrix
\[X_{c,d,n}=\begin{bmatrix}
		0 & a & b  \\
		0 & c & d \\
		1 & 0 & n-c \\
		\end{bmatrix}\]
		where $b=(c-1)(n-c-1)$ and $ad-bc=1$. (In particular, $a$ and $b$ are determined by $c$, $d$ and $n$.) 

 We define an equivalence relation $\sim$ on $\mathcal{C}\mathcal{S}$ generated by $\sim_S$ and $\sim_G$ where
\begin{alignat*}{2}&(c_0,d_0,n)\sim_S  (c_1,d_1,n)\quad&&\textrm{~if~}[\langle \Theta_n-c_0,d_0\rangle]=[\langle \Theta_n-c_1,d_1\rangle]\in C(\Z[\Theta_n]),\\
&(c,d,n)\sim_G(c,d, n+kd)\quad&&\textrm{~if~} k\in \Z.
\end{alignat*}
\begin{conjecture}\label{conjecture:reformulation}For every $(c,d,n)\in \mathcal{C}\mathcal{S}$, $(c,d,n)\sim (1,1,2)$.
\end{conjecture}

\begin{definition}For an integer $n$, we say \emph{Conjecture~\ref{conjecture:reformulation} is true for trace $n$} if for any integers $c$ and $d$ such that $(c,d,n)\in \mathcal{CS}$, $(c,d,n)\sim(1,1,2)$. 
\end{definition}

\begin{remark}By Proposition~\ref{proposition:Aitchison-Rubinstein}, $(c_0,d_0,n)\sim_S(c_1,d_1,n)$ if and only if $X_{c_0,d_0,n}$ and $X_{c_1,d_1,n}$ are similar. The second relation $\sim_G$ corresponds to the equivalence relation $X_{c,d,n}\sim_G X_{c,d,n+kd}$. It is clear that Conjecture~\ref{conjecture:Gompf} for trace $n$ is equivalent to Conjecture~\ref{conjecture:reformulation} since $A_0=X_{1,1,2}$.
\end{remark}

\begin{remark}\label{remark:identity}The pair $(1,1,n+2)\in \mathcal{C}\mathcal{S}$ corresponds to the trivial element of the ideal class monoid $C(\Z[\Theta_{n+2}])$ because $\langle\Theta_{n+2}-1,1\rangle$ is principal. Since $(1,1,n+2)\sim_G(1,1,2)$, we do not have to consider the trivial element of $C(\Z[\Theta_{n+2}])$. (In fact, $X_{1,1,n+2}=A_{n}$ and, as mentioned in the introduction, it has been known that $\Sigma_{A_{n}}^\varepsilon$ is diffeomorphic to $S^4$ for $\varepsilon=0,1$ and $n\in \Z$.)
\end{remark}

	\section{Symmetry between Cappell-Shaneson matrices}\label{section:symmetry}
	In this section, we prove Theorem~\ref{theorem:A} which says that Conjecture~\ref{conjecture:Gompf} for the trace $n$ case is equivalent to the trace $5-n$ case. Throughout this section, let $\Theta_n$ be a root of $f_n(x) =x^3 -n x^2 +(n-1)x -1$ for each integer $n$. We give a ring isomorphism between $\Z[\Theta_n]$ and $\Z[\Theta_{5-n}]$ which will induce a bijection between corresponding ideal class monoids which is compatible with Gompf equivalence. 	

Theorem~\ref{theorem:iso} is inspired by some evidences which are given in work of Aitchison-Rubinstein \cite{Aitchison-Rubinstein:1984-1} and that of Gompf \cite{Gompf:2010-1}. Aitchison and Rubinstein \cite[page~43]{Aitchison-Rubinstein:1984-1} observed that the discriminant $\Delta(f_n)$ of the polynomial $f_n$ have the following symmetry:
		\[
		\Delta(f_n) = n(n-2)(n-3)(n-5)-23=\Delta(f_{5-n}).
		\]
		On the other hand, Gompf \cite[page 1672]{Gompf:2010-1} computed the cardinality $\# C(\mathcal{O}_n)$ for $r \leq 10^8$ 
		via PARI/GP \cite{PARI} and observed that $\# C(\mathcal{O}_n) = \# C(\mathcal{O}_{5-n})$ where $\mathcal{O}_n$ is the ring of integer of $\Q[\Theta_n]$. 
	
	\begin{theorem}\label{theorem:iso}
		For any integer $n$, let $\Theta_n$ be a root of $f_n(x)=x^3-nx^2+(n-1)x-1$. Then, there is a ring isomorphism 
		$\varphi _n : \ztheta{n} \to \ztheta{5-n}$ defined by
		\[\varphi _n (\Theta_{n}) = \Theta _{5-n}^2+(n-4)\Theta_{5-n}+1.\]
	\end{theorem}
	
	\begin{proof}For an aesthetic reason, we prove an equivalent statement that the ring homomorphism $\varphi_{5-n}\colon \Z[\Theta_{5-n}]\to \Z[\Theta_n]$ is an isomorphism for any integer $n$. The ring isomorphism $\varphi_{5-n}$ will be defined as $\varphi_{5-n}(\Theta_{5-n})=\Theta_n^2+(1-n)\Theta_n+1$. Let $\overline{\varphi}_{5-n}\colon \Z[x]\to \Z[\Theta_{n}]$ be a ring homomorphism which sends $x$ to $\Theta_{n}^2+(1-n)\Theta_{n}+1$. We prove that $\overline{\varphi}_{5-n}$ induces the ring homomorphism $\varphi_{5-n}\colon \Z[\Theta_{5-n}]\to \Z[\Theta_n]$ by observing that 
	\[\overline{\varphi}_{5-n}(f_{5-n}(x))=f_{5-n}(\Theta_n^2+(1-n)\Theta_n+1)=0\]
	 where $f_{5-n}(x)=x^3-(5-n)x^2+(4-n)x-1$. By setting $\alpha_n=\varphi_{5-n}(\Theta_{5-n})=\Theta_n^2-(n-1)\Theta_n+1$, we show $f_{5-n}(\alpha_n)=0$. Recall that $\Theta_n$ is a root of $f_n(x)=x^3-nx^2+(n-1)x-1=0$. We have	\[\Theta_n(\Theta_n-1)(\Theta_n-n+1)=\Theta_n^3-n\Theta_n^2+(n-1)\Theta_n=1.\] 
	The following equality will be useful.
	\begin{align*}
	(\Theta_n-1)\alpha_n&=(\Theta_n-1)(\Theta_n^2-(n-1)\Theta_n+1)\\
	&=\Theta_n(\Theta_n-1)(\Theta_n-n+1)+\Theta_n-1=\Theta_n.
	\end{align*}
	Since $\Theta_n-1\neq 0$, the following shows that $f_{5-n}(\alpha_n)=0$:
	\begin{align*}(\Theta_n-1)^3f_{5-n}(\alpha_n)&=(\Theta_n-1)^3(\alpha_n^3-(5-n)\alpha_n^2+(4-n)\alpha_n-1)\\
	&=\Theta_n^3-(5-n)\Theta_n^2(\Theta_n-1)+(4-n)\Theta_n(\Theta_n-1)^2-(\Theta_n-1)^3\\
	&=\Theta_n^3-(\Theta_n-1)^3-\Theta_n(\Theta_n-1)((5-n)\Theta_n-(4-n)(\Theta_n-1))\\
	&=3\Theta_n(\Theta_n-1)+1-\Theta_n(\Theta_n-1)(\Theta_n-n+4)\\
	&=1-\Theta_n(\Theta_n-1)(\Theta_n-(n-1))=0.
	\end{align*}

	Therefore, we have a ring homomorphism $\varphi_{5-n}\colon \Z[\Theta_{5-n}]\to \Z[\Theta_n]$ such that $\varphi_{5-n}(\Theta_{5-n})=\Theta_n^2+(1-n)\Theta_n+1$. Now we prove that $\varphi_{5-n}\circ \varphi_{n}$ is the identity on $\Z[\Theta_n]$ by showing that $\varphi_{5-n}\circ \varphi_{n}(\Theta_n)=\Theta_n$. To simplify the proof, we give two elementary observations. Since $f_{5-n}(\alpha_n)=0$, $\alpha_n(\alpha_n-1)(\alpha_n+n-4)=1$. Note that $(\Theta_n-1)(\alpha_n-1)=(\Theta_n-1)\alpha_n-\Theta_n+1=1$. 
	\begin{align*}\varphi_{5-n}\circ \varphi_{n}(\Theta_n)&=\varphi_{5-n}(\Theta_{5-n}^2+(n-4)\Theta_{5-n}+1)\\
	&=\alpha_n^2+(n-4)\alpha_n+1\\
	&=(\Theta_n-1)(\alpha_n-1)\big(\alpha_n^2+(n-4)\alpha_n+1\big)\\
	&=(\Theta_n-1)\big((\alpha_n-1)\alpha_n(\alpha_n+n-4)+\alpha_n-1\big)\\
	&=(\Theta_n-1)(1+\alpha_n-1)\\
	&=(\Theta_n-1)\alpha_n\\
	&=\Theta_n.
	\end{align*}
	By substituting $n$ by $5-n$, $\varphi_{5-n}\circ \varphi_n$ is also the identity. Hence, $\varphi_{5-n}$ is a ring isomorphism and this completes the proof.
	\end{proof}
	
	\begin{remark0}By tensoring $\Q$ to the ring isomorphism $\varphi_n\colon \Z[\Theta_{n}]\to \Z[\Theta_{5-n}]$ given in Theorem~\ref{theorem:iso}, we obtain a field isomorphism from $\Q[\Theta_{n}]$ to $\Q[\Theta_{5-n}]$. From this field isomorphism, we can see that their ring of integers $\mathcal{O}_n$ and $\mathcal{O}_{5-n}$ are also isomorphic and $\Delta(f_n)=\Delta(f_{5-n})$ for any integer $n$.
		\end{remark0}
	\begin{corollary}\label{corollary:symmetry}There exists a monoid isomorphism $\psi_n\colon C(\Z[\Theta_{n}])\to C(\Z[\Theta_{5-n}])$ for any integer $n$. Furthermore, for any integers $c$, $d$ with $f_n(c)\equiv 0$ \textup{(}mod $d$\textup{)}, $\psi_n$ sends $[\langle \Theta_{n}-c,d\rangle ]$ to $[\langle \Theta_{5-n}-p_{n}(c),d\rangle]$ where $p_{n}(x)=x^2+(1-n)x+1$.
	\end{corollary}
	\begin{proof}Define a monoid homomorphism $\psi_n\colon C(\Z[\Theta_{n}])\to C(\Z[\Theta_{5-n}])$ by $[I]\mapsto [\varphi_n(I)]$ where $\varphi_n$ is the ring homomorphism given in Theorem~\ref{theorem:iso}. Since $\varphi_{n}$ is a ring isomorphism for any integer $n$ by Theorem~\ref{theorem:iso}, $\psi_n$ is also a monoid isomorphism for any integer $n$. To give an explicit formula of $\psi_n$, we prove that $\varphi_n(\langle \Theta_n-c,d\rangle)=\langle  \Theta_{5-n}-p_n(c),d\rangle$ and this clearly implies the desired statement.
	
\begin{claim}$\langle \Theta_{n}-c,d\rangle=\langle p_{n}(\Theta_n)-p_{n}(c),d\rangle$. 
\end{claim}	
\begin{proof}[Proof of Claim]The following calculation shows that $\langle  p_{n}(\Theta_n)-p_{n}(c),d\rangle\subset \langle \Theta_{n}-c,d\rangle$:
	\begin{align*}p_{n}(\Theta_{n})-p_{n}(c)&=\Theta_{n}^2+(1-n)\Theta_{n}+1-c^2-(1-n)c-1\\&=(\Theta_{n}-c)(\Theta_{n}+c-n+1)\in \langle \Theta_{n}-c,d\rangle.\end{align*}
Now we prove $ \langle \Theta_{n}-c,d\rangle\subset \langle  p_{n}(\Theta_n)-p_{n}(c),d\rangle$. Note that 
\[f_{n}(x)+1=x^3-nx^2+(n-1)x=(x-1)(x^2+(1-n)x)=(x-1)(p_{n}(x)-1).\]
Using the above equation on $f_n(x)+1$ and the fact that $f_n(\Theta_n)=0$, we observe that
		\begin{align*}
		\Theta_{n} -c &= \Theta_{n}-1-(c-1)\\
		&=(\Theta_{n}-1)(f_{n}(c)+1)-(\Theta_{n}-1)f_{n}(c)-(c-1)(f_{n}(\Theta_{n})+1)\\
		&=(\Theta_n-1)(c-1)(p_n(c)-1)-(\Theta_n-1)f_n(c)-(\Theta_n-1)(c-1)(p_n(\Theta_n)-1)\\
			&=(\Theta_{n}-1)(c-1)(p_{n}(c)-p_{n}(\Theta_{n}))-(\Theta_{n}-1)f_n(c).
		\end{align*}
From the assumption $f_n(c)\equiv 0$ (mod $d$), $f_n(c) \in \langle p_{n}(\Theta_n)-p_{n}(c),d\rangle$.  This shows that $\langle  \Theta_n-c,d\rangle\subset \langle p_{n}(\Theta_n)-p_{n}(c),d\rangle$ and hence the claim follows.
\end{proof}
	 Note that $p_{n}(\Theta_{n})=\Theta_{n}^2+(1-n)\Theta_{n}+1=\varphi_{5-n}(\Theta_{n})$. By the claim,
\[\varphi_n(\langle \Theta_{n}-c,d\rangle)=\varphi_n(\langle \varphi_{5-n}(\Theta_{n})-p_{n}(c),d\rangle)=\langle \Theta_{n}-p_{n}(c),d\rangle.\]
Here the last equality follows from the fact that  $\varphi_n\circ \varphi_{5-n}$ is the identity. This completes the proof.
	\end{proof}

	\begin{theorem}\label{sym1}
		There is a bijection between the set of similarity classes of 
		Cappell-Shaneson matrices with trace $n$ and the set of similarity classes of \CS matrices with trace $5-n$, 
		which is explicitly defined by
		\[A=
		\begin{bmatrix}
		0 & a & b \\
		0 & c & d \\
		1 & 0 & n-c \\
		\end{bmatrix}		\mapsto  A^*=
		\begin{bmatrix}
		0 & a^*& b^* \\
		0 & c^* & d^* \\
		1 & 0 & 5-n-c^* \\
		\end{bmatrix}
		\]
		where $c^*=p_n(c)=c^2+(1-n)c+1$, $d^*=d$. In particular, $X_{c,d,n}^*=X_{p_n(c),d,5-n}^{\vphantom{*}}$.
	\end{theorem}
	
	\begin{proof}Since every \CS matrix is similar to a standard \CS matrix, the bijection $A\mapsto A^*$ gives the ones
	which represent all the similarity classes of \CS matrices with trace $5-n$.
	
	Recall that there is a bijection between the similarity classes of Cappell-Shaneson matrices with trace $n$  (respectively, trace $5-n$) with the ideal class monoid $C(\Z[\Theta_n])$ (respectively, $C(\Z[\Theta_{5-n}])$) by Proposition~\ref{proposition:Aitchison-Rubinstein}. On the other hand, we have a monoid isomorphism $\psi_n\colon C(\Z[\Theta_n])\to C(\Z[\Theta_{5-n}])$ by Corollary~\ref{corollary:symmetry}. The composition of these three bijections gives a bijection between the set of similarity classes of Cappell-Shaneson matrices with trace $n$ and the set of similarity classes of Cappell-Shaneson matrices with trace $5-n$. It remains to show is that the aforementioned bijection actually sends a standard Cappell-Shaneson matrix $A$ to a standard Cappell-Shaneson matrix $A^*$. 
	
By Proposition~\ref{proposition:Aitchison-Rubinstein}, the ideal class correspond to the standard Cappell-Shaneson matrix $A$ is $[\langle \Theta_n-c,d\rangle]$. By Corollary~\ref{corollary:symmetry}, $\psi_n$ sends the ideal class $[\langle \Theta_n -c,d\rangle]$ to the ideal class $[\langle  \Theta_{5-n}-p_{n}(c),d\rangle]$, which is the ideal class correspond to the standard Cappell-Shaneson matrix $A^*$ by Proposition~\ref{proposition:Aitchison-Rubinstein}. This completes the proof.
	\end{proof}
	
\begin{example}\label{ten}As an illustration, we explicitly describe the bijection given in Theorem~\ref{sym1} for the case that trace $n=-5$.	 Aitchison and Rubinstein \cite{Aitchison-Rubinstein:1984-1} showed that there are only two similarity classes 
		of \CS matrices with trace $-5$, which are represented by as follows. (Note that $A=A_{-7}$.)
		\[
		A=
		\begin{bmatrix}
		0 & 1 & 0 \\
		0 & 1 & 1 \\
		1 & 0 & -6 \\
		\end{bmatrix},\,
		B=
		\begin{bmatrix}
		0 & -5 & -8 \\
		0 & 2 & 3 \\
		1 & 0 & -7 \\
		\end{bmatrix}.
		\]
		By Theorem~\ref{sym1}, it follows that there are only two similarity classes 
		of \CS matrices with trace $10$, which are represented by 
				\[
		A^*=
		\begin{bmatrix}
		0 & 57 & 7 \\
		0 & 8 & 1 \\
		1 & 0 & 2 \\
		\end{bmatrix},\,
		B^*=
		\begin{bmatrix}
		0 & -725 & -128 \\
		0 & 17 & 3 \\
		1 & 0 & -7 \\
		\end{bmatrix}.
		\]
		By Proposition~\ref{proposition:Aitchison-Rubinstein}, the ideal classes correspond to $A^*$ and $B^*$ are 
		$[\langle 1, \Theta_{10} -8\rangle]$ and $[\langle 3, \Theta_{10}-17\rangle]$, respectively.
		Note that $[\langle 1, \Theta_{10} -8\rangle =[\langle 1, \Theta_{10}-1\rangle]$ and $[\langle 3, \Theta_{10}-17\rangle]=[\langle 3, \Theta_{10}-2\rangle ]$. 
		We obtain similarity relations by Proposition~\ref{proposition:Aitchison-Rubinstein}:
		\[
		A^*\sim 
		\begin{bmatrix}
		0 & 1 & 0 \\
		0 & 1 & 1 \\
		1 & 0 & 9 \\
		\end{bmatrix}, \,
		B^*\sim 
		\begin{bmatrix}
		0 & 5 & 7  \\
		0 & 2 & 3 \\
		1 & 0 & 8 \\
		\end{bmatrix}.
		\]
	\end{example}	
	To complete the proof of Theorem~\ref{theorem:A}, we prove two lemmas which illustrate  that the bijection given in Theorem~\ref{sym1} behaves nicely with Gompf equivalence.
	\begin{lemma}\label{lemma:Gompfdual}Suppose that $A$ and $B$ are two standard \CS matrices such that $A$ and $B$ are Gompf equivalent. Then $A^*$ and $B^*$ are also Gompf equivalent.
	\end{lemma}
	\begin{proof} Since Gompf equivalence is generated by $\sim_S$ and $\sim_G$, we can assume without loss of generality that either $A\sim_S B$ or $A\sim_G B$ holds. If $A\sim_S B$, then $A^*\sim_S B^*$ by Theorem~\ref{sym1}. Now we assume that $A\sim_G B$, that is, $A=X_{c,d,n}$ and $B=X_{c,d,n+kd}$ for some $c,d,k$ and $n\in \Z$. By Theorem~\ref{sym1}, $A^*=X_{c^*,d,5-n}$ and $B^*=X_{c^*,d,5-n-kd}$, and hence $A^*\sim_G B^*$. (Note that $d^*=d$.) This completes the proof.
	\end{proof}
	\begin{lemma}\label{lemma:Gompfdoubledual}Let $A$ be a standard \CS matrix. Then $A$ is similar to $(A^*)^*$.
	\end{lemma}
	\begin{proof}
	
	Since $A$ is a standard \CS matrix, $A=X_{c,d,n}$ for some $c,d$ and $n$ with $f_n(c)\equiv 0\Mod d$.	Then $(A^*)^*=(X_{p_n(c),d,5-n})^*=X_{p_{5-n}(p_n(c)),d,n}$ where $p_n(c)=c^2+(1-n)c+1$. Note that 
	\begin{align*}p_{5-n}(p_n(c))&=p_n(c)^2+(n-4)p_n(c)+1\\
	&=\big(c^2+(1-n)c+1\big)^2+(n-4)\big(c^2+(1-n)c+1\big)+1\\
	&=c+f_n(c)(c-n+2)\equiv c\Mod{d}
	\end{align*}
	since $f_n(c)\equiv 0\Mod d$. Since $p_{5-n}(p_{n}(c))\equiv c\Mod d$, $(A^*)^*$ is similar to $A$ by Remark~\ref{remark:changenbyc}.
	\end{proof}
	Now we prove Theorem~\ref{theorem:A}.
	\begin{proof}[Proof of Theorem~\ref{theorem:A}] We have already seen that there is a bijection between the set of similarity classes of trace $n$ \CS matrices and the set of similarity classes of trace $5-n$ \CS matrices in Theorem~\ref{sym1}.
	
	Assume that Conjecture~\ref{conjecture:Gompf} is true for trace $n$ for some integer $n$. Let $X$ be a standard \CS matrix with trace $5-n$. (Recall that every \CS matrix is similar to a standard \CS matrix.) Then $X^*$ given in Theorem~\ref{sym1} is a standard \CS matrix with trace $n$. Since we are assuming that Conjecture~\ref{conjecture:Gompf} is true for trace $n$, $X^*_{\vphantom{0}}$ is Gompf equivalent to $A_0$. By Lemma~\ref{lemma:Gompfdoubledual}, $X$ is similar to $(X^*)^*$. By Lemma~\ref{lemma:Gompfdual}, $(X^*)^*$ is Gompf equivalent to $(A_0)^*$. As in Example~\ref{ten}, $A_0^*$ is similar to $A_1$, which is Gompf equivalent to $A_0$ by Remark~\ref{remark:identity}.  Therefore $X$ is Gompf equivalent to $A_0$. This shows that Conjecture~\ref{conjecture:Gompf} is true for trace $5-n$ if Conjecture~\ref{conjecture:Gompf} is true for trace $n$. This completes the proof.
	\end{proof}

\section{Ideal class monoid $C(\Z[\Theta_n])$}\label{section:idealclassmonoid}

	In this section, we use several techniques from algebraic number theory. We will recall Dedekind-Kummer theorem, and show $C(\Z[\Theta_{49k+27}])$ is not a group for any integer $k$. We will also determine the structure of the ideal class monoid $C(\Z[\Theta_{27}])$. We first collect some definitions following  \cite{Stevenhagen:2008-1}. 
	
	\begin{definition}[Number rings and orders]A \emph{number field} $K$ is a finite degree field extension of the field $\Q$ of rational numbers. A \emph{number ring} is an integral domain $R$ for which the field of fractions $K$ is a number field. For a number field $K$ with degree $n$, a subring $R$ of the number field $K$ is called an \emph{order} if $R$ is a free $\Z$-module of rank $n$.
	\end{definition}
	
	\begin{example}[{$\Z[\Theta_n]$ is an order}] Let $\alpha$ be a root of some monic, irreducible polynomial $f\in \Z[x]$ of degree $n$. Then $\Q[\alpha]$ is a number field of degree $n$. The ring $\Z[\alpha]$  obtained by adjoining to $\Z$ has a free $\Z$-basis $1,\alpha,\ldots,\alpha^{n-1}$ and hence $\Z[\alpha]$ is an order in the number field $\Q[\alpha]$. We are principally interested in the orders of the form $\Z[\Theta_n]$ where $\Theta_n$ is a root of the monic, irreducible polynomial $f_n(x)=x^3-nx^2+(n-1)x-1$. 
    \end{example}

	\begin{definition}[Ring of integers] Let $K$ be a number field. An element $x$ in $K$ is \emph{an integral element} if $x$ is a root of monic, irreducible polynomial with integer coefficients. The set of integral elements in $K$ is called \emph{the ring of integer} of $K$ and denoted by $\mathcal{O}_K$.
	\end{definition}
	We recall elementary facts on orders discussed in \cite{Stevenhagen:2008-1}.
	\begin{theorem}[{\cite[Sections 6--7]{Stevenhagen:2008-1}}]\label{theorem:Dedekind} A number ring $R\subset K$ is an order in $K$ if and only if $R$ is of finite index in $\mathcal{O}_K$. In particular, $\mathcal{O}_K$ is the maximal order in $K$. For an order $R\subset K$, the following conditions are equivalent.
	\begin{enumerate}
	\item $R$ is integrally closed.
	\item  $R$ is the maximal order $\mathcal{O}_K$. 
	
	\item $R$ is a Dedekind domain.
	\item Every ideal of $R$ is invertible.
	\item $C(R)$ is a group.
	\end{enumerate}
	\end{theorem}

	\subsection{Dedekind-Kummer theorem}\label{subsection:Dedekind-Kummer}
	As in Section~\ref{subsection:Latimer-MacDuffee-Taussky}, for two ideals $I$ and $J$ in $\Z[\Theta_n]$, we say $I$ and $J$ are \emph{equivalent} (and denoted by $I\approx J$) if $\alpha I=\beta J$ for some non-zero $\alpha, \beta\in \Z[\Theta_n]$. By the definition of $C(\Z[\Theta_n])$, $I\approx J$ if and only if  $[I]=[J]\in C(\Z[\Theta_n])$. By Proposition~\ref{proposition:Aitchison-Rubinstein}, every ideal of $\Z[\Theta_n]$ is equivalent to $\langle \Theta_n-c,d\rangle$ for some $c,d\in \Z$ such that $f_n(c)\equiv 0\Mod d$. By Proposition~\ref{proposition:49k+27}, we know that there are infinitely many $n$ such that $C(\Z[\Theta_n])$ is not a group. For those $n$, there is a non-invertible ideal $\langle \Theta_n-c,d\rangle$ of $\Z[\Theta_n]$. Therefore we want to determine when the ideal $\langle \Theta_n-c,d\rangle$ such that $f_n(c)\equiv 0 \Mod{d}$ is invertible. For this purpose, we can assume that $d$ is prime power by the following remark.
	\begin{remark}\label{remark:decomposition}
	Suppose that  $p$ and $q$ are relatively prime integers. Then $\Theta_n-c$ is a linear combination of $p(\Theta_n-c)$ and $q(\Theta_n-c)$. It follows that
	\[\langle \Theta_n-c,p\rangle \langle \Theta_n-c,q\rangle=\langle (\Theta_n-c)^2, p(\Theta_n-c),q(\Theta_n-c),pq\rangle =\langle \Theta_n-c,pq\rangle.\]
	More generally, consider the prime factorization $d=p_1^{e_1}\cdots p_m^{e_m}$. Then 
	 \[\langle \Theta_n-c,d\rangle =\langle \Theta_n-c,p_1^{e_1}\rangle\langle \Theta_n-c,p_2^{e_2}\rangle \cdots \langle \Theta_n-c,p_m^{e_m}\rangle.\]
Note that $f_n(c)\equiv 0 \Mod {p_i^{e_i}}$ since $f_n(c)\equiv 0 \Mod d$. 	
\end{remark}	
	
Following \cite[Theorem 8.2]{Stevenhagen:2008-1}, we recall Dedekind-Kummer theorem, which can be used to determine when the ideal of the form $\langle \Theta_n-c,p\rangle$ with $p$ is prime and $f_n(c)\equiv 0\Mod{p}$ is invertible.	
	\begin{theorem}[Dedekind-Kummer {\cite[Theorem 8.2]{Stevenhagen:2008-1}}]\label{theorem:Dedekind-Kummer}Let $p$ be a prime integer and $\alpha$ be a root of a monic, irreducible polynomial $f(x)\in \Z[x]$. Let $\overline{f}\in \Z_p[x]$ be a polynomial such that $\overline{f}\equiv f\Mod{p}$. Let the factorization of $\overline{f}$ in $\Z_p[x]$ be $\prod_{i=1}^l\overline{g}_i^{e_i}$. Let $g_i\in \Z[x]$ be a polynomial such that $g_i\equiv \overline{g}_i\Mod{p}$. If $r_i\in \Z[x]$ is the remainder of $f$ upon division by $g_i$ in $\Z[x]$, that is, $f=g_iq_i+r_i$, then the ideal $ \mathfrak{p}_i=\langle p, g_i(\alpha)\rangle \subset \Z[\alpha]$ is prime and $\mathfrak{p}_i$ is invertible if and only if at least one of the following conditions holds.
	\begin{enumerate}
	\item $e_i=1$.
	\item $p^2$ does not divide $r_i\in \Z[x]$.
	\end{enumerate}
	\end{theorem}
	By applying Dedekind-Kummer theorem to the case that $\alpha=\Theta_n$ and $f=f_n(x)$, we obtain the following proposition which gives a simple, but complete characterization when ideals of the form $\langle \Theta_n-c,p\rangle$ with $p$ is prime and $f_n(c)\equiv 0 \Mod{p}$ are invertible. This will be useful in our analysis of the structure of $C(\Z[\Theta_{27}])$.

	\begin{proposition}\label{proposition:invertible-prime}Suppose that integers $c$, $n$ and $p$ satisfy $f_n(c)\equiv 0 \Mod {p}$. If $p$ is prime, then $\langle \Theta_n-c,p\rangle$ is a prime ideal of $\Z[\Theta_n]$. The ideal $\langle \Theta_n-c,p\rangle$ is invertible 
	 if and only if at least one of the following conditions holds.
	\begin{enumerate}
	\item $c$ is a simple root of $f_n(x)$ modulo $p$. 
	\item $p^2$ does not divide $f_n(c)$.
	\end{enumerate}
	\end{proposition}
	\begin{proof}[Proof of Proposition~\ref{proposition:invertible-prime}]
Recall $f_n(x)=x^3-nx^2+(n-1)x-1$ is a monic, irreducible polynomial with a root $\Theta_n$. If $f_n(c)\equiv 0\Mod{p}$, then $x-c$ is a factor of $\overline{f_n}$ in $\Z_p[x]$. On the other hand, we can write $f_n(x)=(x-c)q(x)+f_n(c)$. By applying Theorem~\ref{theorem:Dedekind-Kummer} for $\mathfrak{p}=\langle p, \Theta_n-c\rangle$ where $g(x)=x-c$ and $r(x)=f_n(c)$, we obtain the  conclusion.
	\end{proof}

	\begin{proposition}\label{proposition:invertible-primepower}Suppose that $p$ is a prime integer and an integer $c$ satisfies $f_n(c)\equiv 0\Mod {p^k}$ for some positive integer $k$. 
	\begin{enumerate}
	\item If $\langle \Theta_n-c,p\rangle $ is invertible, then $\langle \Theta_n-c,p^k\rangle$ is invertible. 
	\item If $f_n(c)\not\equiv 0\Mod{p^{k+1}}$, then $\langle \Theta_n-c,p^k\rangle$ is invertible.
	\end{enumerate}
	\end{proposition}
	\begin{proof}(1) Denote $I=\langle \Theta_n-c,p^k\rangle$ and $\mathfrak{p}=\langle \Theta_n-c,p\rangle$. Assume that $\mathfrak{p}$ is invertible. We first observe that $\sqrt{I}=\mathfrak{p}$ where $\sqrt{I}$ is the radical of $I$. Let $\alpha$ be an element in $\mathfrak{p}$. We can write $\alpha=xp+y(\Theta_n-c)$ for some $x,y\in \Z[\Theta_n]$. Then $\alpha^k=(xp+y(\Theta_n-c))^k\in I$. This shows that $\mathfrak{p}\subset \sqrt{I}$. Recall that $\sqrt{I}$ is the intersection of all prime ideals which contain $I$. By Proposition~\ref{proposition:invertible-prime}, $\mathfrak{p}$ is a prime ideal which contains $I$. It follows that $\sqrt{I}\subset\mathfrak{p}$. Since we are assuming $\mathfrak{p}$ is invertible, by Lemma~\ref{lemma:commutativealgebra} below, we conclude that $I$ is invertible.

	(2) From the hypothesis, we can write $f_n(c)=p^k\cdot q$ where $q$ is relatively prime to $p$. Then, $\langle \Theta_n-c,p^k\rangle \langle \Theta_n-c,q\rangle =\langle \Theta_n-c,p^k\cdot q\rangle$ by Remark~\ref{remark:decomposition}. Since $f_n(\Theta_n)=0$, we have $p^k\cdot q =f_n(c)-f_n(\Theta_n)\in \langle \Theta_n-c\rangle$. It follows that $\langle \Theta_n-c,p^k\cdot q\rangle=\langle \Theta_n-c\rangle$ which is a principal ideal. This completes the proof.
	\end{proof}

	\begin{lemma}\label{lemma:commutativealgebra}Let $R$ be a number ring and $\mathfrak{p}$ be an invertible prime ideal of $R$. If $I$ is an ideal of $R$ such that $\sqrt{I}=\mathfrak{p}$, then $I=\mathfrak{p}^k$ for some $k$. In particular, $I$ is an invertible ideal.
	\end{lemma}

	\begin{proof}By \cite[page 213]{Stevenhagen:2008-1}, every ideal of $R$ is finitely generated, and every prime ideal of $R$ is maximal. In particular, $\mathfrak{p}$ is maximal. By \cite[Proposition 4.2]{Atiyah-Macdonald}, $I$ is $\mathfrak{p}$-primary. That is, $\sqrt{I}=\mathfrak{p}$, and if $xy\in I$, then either $x\in I$ or $y^n\in I$ for some $n>0$.
	
	Let $K$ be the quotient field of $R$. Consider the localization
	$R_\mathfrak{p}=\{\tfrac{r}{s}\in K\mid r\in R, s\notin \mathfrak{p}\}$ of $R$ at~$\mathfrak{p}$ and the canonical homomorphism $f_\mathfrak{p}\colon R\to R_{\mathfrak{p}}$. Since $\mathfrak{p}$ is an invertible prime ideal, every ideal of $R_\mathfrak{p}$ is a power of $\mathfrak{p}R_\mathfrak{p}$ by \cite[Proposition 5.4]{Stevenhagen:2008-1}. In short, $R_\mathfrak{p}$ is a discrete valuation ring. 
	
	Let $I_\mathfrak{p}$ be the extension of $I$. That is, $I_\mathfrak{p}$ is the ideal of $R_\mathfrak{p}$ generated by $f_\mathfrak{p}(I)$. Since $R_\mathfrak{p}$ is a discrete valuation ring, $I_\mathfrak{p}=(\mathfrak{p}R_\mathfrak{p})^k$ for some $k$. Then 
	\[I=f_\mathfrak{p}^{-1}(I_\mathfrak{p})=f_{\mathfrak{p}}^{-1}((\mathfrak{p}R_\mathfrak{p})^k)=(f_{\mathfrak{p}}^{-1}(\mathfrak{p}R_\mathfrak{p}))^k=\mathfrak{p}^k.\]
	We remark that the first equality uses the fact that $I$ is $\mathfrak{p}$-primary (see \cite[Proposition 3.11(2) and Lemma 4.4(3)]{Atiyah-Macdonald}). This completes the proof. 
		\end{proof}
	
	\subsection{The ideal class monoid $C(\Z[\Theta_{49k+27}])$ is not a group}\label{subsection:49k+27}
	In this subsection, we prove that there are infinitely many integers $n$ such that $\Z[\Theta_{n}]$ is not a Dedekind domain. Hence, to study general Cappell-Shaneson spheres, we need to understand equivalence classes of non-invertible ideals of $\Z[\Theta_n]$ for those $n$. For general $n$, finding an explicit  formula for $\# C(\Z[\Theta_n])$ (and its representatives) seems to be a difficult problem in algebraic number theory. Because of these subtleties, proving Conjecture~\ref{conjecture:Gompf} is difficult.
	 
	 \begin{proposition}\label{proposition:49k+27} For any integer $k$, the ideal $\langle \Theta_{49k+27}-2,7\rangle$ is not an invertible ideal in $\Z[\Theta_{49k+27}]$, and hence $C(\Z[\Theta_{49k+27}])$ is not a group. Consequently, $\Z[\Theta_{49k+27}]$ is not a Dedekind domain for any integer $k$.
	 \end{proposition}
	 \begin{proof}We first observe that 
	 \[f_{49k+27}(x)=x^3-(49k+27)x^2+(49k+26)x-1\equiv(x-2)^3\Mod{7}.\]
	 It is straightforward to check that $f_{49k+27}(2)=-49(2k+1)$. Therefore, $2$ is not a simple root of $f_{49k+27}(x)\equiv 0\Mod 7$, and $49$ divides $f_{49k+27}(2)$.	 By Proposition~\ref{proposition:invertible-prime}, $\langle \Theta_{49k+27}-2,7\rangle$ is not an invertible ideal of $\Z[\Theta_{49k+27}]$ for any integer $k$.
	 \end{proof}

	 Recall from Theorem~\ref{theorem:Dedekind} that $C(R)$ is not a group if and only if $R$ is not integrally closed. We give another proof of the fact that $C(\Z[\Theta_{49k+27}])$ is not a group by directly showing that $\Z[\Theta_{49k+27}]$ is not integrally closed for any integer $k$.

		\begin{proposition}\label{proposition:49k+27-2}
		For any integer $k$, $\ztheta{49k+27}$ is not integrally closed, and hence $C(\Z[\Theta_{49k+27}])$ is not a group for any integer $k$. Equivalently, $\Z[\Theta_{49k+27}]$ is not a Dedekind domain for any integer~$k$.
	\end{proposition}
	
	\begin{proof}Fix an integer $k$ and it suffices to prove that $\ztheta{49k+27}$ is a proper subset of $\mathcal{O}_{49k+27}$ where $\mathcal{O}_{49k+27}$ is the ring of integer of $\Z[\Theta_{49k+27}]$. Set $\eta_k = \frac{1}{7}(\Theta_{49k+27}-2)^2 \in \qtheta{49k+27}$.
		We will show that $\eta_k$ is an integral element or equivalently $\eta_k\in \mathcal{O}_{49k+27}$.
		Let 
		\begin{align*}
		g_k(x)&= x^3 -(343k^2+336k+83)x^2 + (245k^2+238k+58)x -(28k^2+28k+7), \\
		u(x) &= \tfrac{1}{7}(x-2)^2.
		\end{align*}
		Then $g_k(\eta_k)=g_k(u(\Theta_{49k+27}))=0$ since
		\[343g_k(u(x))=f_{49k+27}(x)(x^3+(49k+15)x^2-(343k+142)x+588k+265).\]
		The last equality can be easily checked by expanding terms in both sides.	
	
Now we prove that $g_k$ is irreducible over $\Z$ for any integer $k$. Suppose that the cubic, monic polynomial $g_k$ is reducible over $\Z$. Then $g_k$ is reducible over $\Z_2$ so it has a solution in $\Z_2$. Since $g_k(0)$ and $g_k(1)$ are odd, $g_k$ does not have a solution in $\Z_2$. It follows that $g_k$ is irreducible over $\mathbb{Z}$. 
		That is, $g_k$ is the minimal polynomial of $\eta_k$, and
		hence $\eta_k \in \mathcal{O}_{49k+27}$.
		Since $\eta_k \not \in \ztheta{49k+27}$, this completes the proof that $\ztheta{49k+27}$ is not equal to $\mathcal{O}_{49k+27}$.
	\end{proof}

	\subsection{The computation of the ideal class monoid $C(\Z[\Theta_{27}])$}\label{subsection:27}
	
	In the previous section, we showed that $C(\Z[\Theta_{49k+27}])$ is not a group for any integer $k$. Among $3\leq n \leq 75$, $n=27$ is the only case that $C(\Z[\Theta_n])$ is not a group, but a monoid (this can be checked either using MAGMA or PARI/GP). Nonetheless, MAGMA can still compute the Picard group $\operatorname{Pic}(\Z[\Theta_{27}])$ consists of the ideal  classes of \emph{invertible} ideals in $\Z[\Theta_{27}]$ (see Section~\ref{subsection:27MAGMA}).
	
	 The goal of this subsection is to prove Theorem~\ref{theorem:idealclassmonoid27} where we determine the monoid structure of $C(\Z[\Theta_{27}])$. The key ingredients are Proposition~\ref{proposition:27} and the computation of  $\operatorname{Pic}(\Z[\Theta_{27}])$.

	\begin{proposition}\label{proposition:27} Let $I$ be a non-zero ideal of $\Z[\Theta_{27}]$. Then exactly one of the following holds.
	\begin{enumerate}
	\item $I$ is invertible.
	\item $I$ is equivalent to $\langle \Theta_{27}-2,7\rangle\cdot J$ for some invertible ideal $J$.
	\end{enumerate}
	\end{proposition}
	\begin{proof}Let $I$ be a non-zero ideal of $\Z[\Theta_{27}]$. By Proposition~\ref{proposition:Aitchison-Rubinstein}, $I$ is equivalent to $\langle \Theta_{27}-c,d\rangle$ where $f_{27}(c)\equiv 0\Mod {d}$. Consider the prime factorization $d=p_1^{e_1}p_2^{e_2}\cdots p_m^{e_m}$. By Remark~\ref{remark:decomposition},
	\[\langle \Theta_{27}-c,d\rangle =\langle \Theta_{27}-c,p_1^{e_1}\rangle\langle \Theta_{27}-c,p_2^{e_2}\rangle \cdots \langle \Theta_{27}-c,p_m^{e_m}\rangle.\]
	Since the product of invertible ideals is invertible, it suffices to  determine which ideals $\langle \Theta_{27}-c,p_i^{e_i}\rangle$ are not invertible. (Note that $f_{27}(c)\equiv 0 \Mod {p_i^{e_i}}$ for any $i$.)
	
		Recall that Aitchison and Rubinstein \cite[page~43]{Aitchison-Rubinstein:1984-1} computed the discriminant 
		\[\Delta(f_{27})=27\cdot 25\cdot 24\cdot 22-23=356377\] which has the prime factorization $7^3\cdot 1039$. If a prime $p$ does not divide the discriminant $\Delta(f_{27})$, then every root of $f_{27}(x)$ modulo $p$ is a simple root. 
	By Propositions~\ref{proposition:invertible-prime} and~\ref{proposition:invertible-primepower}, if the prime $p_i$ is not equal to $7$ and $1039$, then $\langle \Theta_{27}-c,p_i^{e_i}\rangle$ is invertible.
	
Consider an ideal $\langle \Theta_{27}-c,1039^k\rangle$ such that $f_{27}(c)\equiv 0\Mod{{1039}^k}$. We show that $\langle \Theta_{27}-c,1039\rangle$ is invertible. By Proposition~\ref{proposition:invertible-primepower}, this implies that $\langle \Theta_{27}-c,1039^k\rangle$ is invertible. Since 
\[f_{27}(x)=x^3-27x^2+26x-1\equiv(x-453)^2(x-160) \Mod{1039},\]
$c=1039l+453$ or $c=1039l+160$. Since 160 is a simple root of $f_{27}(x)\equiv 0\Mod{1039}$, $\langle \Theta_{27}-c,160\rangle$ is invertible by Proposition~\ref{proposition:invertible-prime}. On the other hand, consider the prime factorization $f_{27}(453)=13\cdot 1039\cdot 6473$. By Proposition~\ref{proposition:invertible-prime}, this shows that $\langle \Theta_{27}-c,453\rangle$ is also invertible. 
	
	Now it remains to consider an ideal $\langle \Theta_{27}-c,7^k\rangle$ such that $f_{27}(c)\equiv 0\Mod{{7}^k}$. If $f_{27}(c)\not\equiv 0 \Mod{7^{k+1}}$, then $\langle \Theta_{27}-c,7^k\rangle$ is invertible by Proposition~\ref{proposition:invertible-primepower}(2). If $f_{27}(c)\equiv 0\Mod{7^{k+1}}$, then $\langle \Theta_{27}-c,7^k\rangle\approx \langle \Theta_{27}-2,7\rangle$ by Lemma~\ref{lemma:7} below. This completes the proof. 
	\end{proof}
	\begin{lemma}\label{lemma:7}Let $k$ be a positive integer and $c$ be an integer such that $f_{27}(c)\equiv 0 \Mod{{7^{k+1}}}$. Then $\langle \Theta_{27}-c,7^k\rangle\approx \langle \Theta_{27}-2,7\rangle$. 
	\end{lemma}
	\begin{proof}[Proof of Lemma~\ref{lemma:7}] By Proposition~\ref{proposition:cmod7}, $c\equiv 2\Mod{7}$. If $k=1$, then $\langle \Theta_{27}-c,7\rangle=\langle\Theta_{27}-2,7\rangle$ since $c\equiv 2\Mod{7}$.
	
	If $k\geq 2$, we apply Proposition~\ref{proposition:claim1} several times to  obtain the desired conclusion
	\begin{center}$\langle \Theta_{27}-c,7^k\rangle \approx \langle \Theta_{27}-c,7^{k-1}\rangle \approx\cdots \approx\langle \Theta_{27}-c,7\rangle=\langle \Theta_{27}-2,7\rangle$.\end{center}
The last equality follows because $c\equiv 2\Mod{7}$.\end{proof}
	\begin{proposition}\label{proposition:cmod7}Let $c$ be an integer such that $f_{27}(c)\equiv 0\Mod{7}$. Then $c\equiv2\Mod {7}$.
	\end{proposition}
	\begin{proof}Since $f_{27}(x)=x^3-27x^2+26x-1\equiv (x-2)^3\Mod 7$, the conclusion directly follows.
	\end{proof}

	\begin{proposition}\label{proposition:claim1}If $k\geq 2$ and $f_{27}(c)\equiv 0 \Mod{{7^{k+1}}}$, then $\langle \Theta_{27}-c,7^k\rangle\approx \langle \Theta_{27}-c,7^{k-1}\rangle$. 
	\end{proposition}
	\begin{proof}Note that
	\[f_{27}(x)=x^3-27x^2+26x-1=(x-2)^3-7(3(x-2)^2+10(x-2)+7).\]  
	Since $\Theta_{27}$ is a root of $f_{27}(x)$, $(\Theta_{27}-2)^3=7(3(\Theta_{27}-2)^2+10(\Theta_{27}-2)+7)$. By Proposition~\ref{proposition:cmod7}, we can write $c=7l+2$ for some $l\in\Z$. 
It follows that 
\[(\Theta_{27}-c)(\Theta_{27}-2)^2=(\Theta_{27}-2)^3-7l(\Theta_{27}-2)^2=7((3-l)(\Theta_{27}-2)^2+10(\Theta_{27}-2)+7).\]
	
In Proposition~\ref{proposition:claim2}, if $k\geq 2$, then we will observe that 
\[\langle \Theta_{27}-c,7^k\rangle=\langle (3-l)(\Theta_{27}-2)^2+10(\Theta_{27}-2)+7,7^{k-2}(\Theta_{27}-2)^2\rangle.\]
 This observation completes the proof since 
	\begin{align*}\langle \Theta_{27}-c,7^{k-1}\rangle&\approx \langle (\Theta_{27}-c)(\Theta_{27}-2)^2,7^{k-1}(\Theta_{27}-2)^2\rangle\\
	&=7\langle (3-l)(\Theta_{27}-2)^2+10(\Theta_{27}-2)+7,7^{k-2}(\Theta_{27}-2)^2\rangle\\&\approx \langle (3-l)(\Theta_{27}-2)^2+10(\Theta_{27}-2)+7,7^{k-2}(\Theta_{27}-2)^2\rangle\\
	&=\langle \Theta_{27}-c,7^k\rangle\end{align*}
where we have used the equality $(\Theta_{27}-c)(\Theta_{27}-2)^2=7((3-l)(\Theta_{27}-2)^2+10(\Theta_{27}-2)+7)$.
\end{proof}

\begin{proposition}\label{proposition:claim2}If $k\geq 2$ and $f_{27}(c)\equiv 0 \Mod{{7^{k+1}}}$, then $\langle \Theta_{27}-c,7^k\rangle=\langle (3-l)(\Theta_{27}-2)^2+10(\Theta_{27}-2)+7,7^{k-2}(\Theta_{27}-2)^2\rangle$. 
\end{proposition}
\begin{proof}Recall that 
\[f_{27}(x)=x^3-27x^2+26x-1=(x-2)^3-7(3(x-2)^2+10(x-2)+7).\]
By Proposition~\ref{proposition:cmod7}, we can write $c=7l+2$ for some $l\in\Z$. Then $f_{27}(c)=49(7l^3-21l^2-10l-1)$. Since $f_{27}(c)\equiv 0 \Mod{7^{k+1}}$ and $k\geq 2$, \[7l^3-21l^2-10l-1\equiv 0\Mod{7^{k-1}}.\] 

We first prove that $\langle (3-l)(\Theta_{27}-2)^2+10(\Theta_{27}-2)+7,7^{k-2}(\Theta_{27}-2)^2\rangle\subset \langle \Theta_{27}-c,7^k\rangle$.
Since $k\geq 2$, the following computation shows that $7^{k-2}(\Theta_{27}-2)^2\in \langle \Theta_{27}-c,7^k\rangle$:
	\begin{align*}7^{k-2}(\Theta_{27}-2)^2&=7^{k-2}(\Theta_{27}-c+7l)^2=(\Theta_{27}-c)\left(7^{k-2}(\Theta_{27}-c)+2\cdot 7^{k-1}l\right)+7^kl^2.\end{align*}
	Since $c=7l+2$,  $(3-l)(\Theta_{27}-2)^2+10(\Theta_{27}-2)+7=(3-l)(\Theta_{27}-c+7l)^2+10(\Theta_{27}-c+7l)+7$. The right hand side is equal to $(\Theta_{27}-c)\big((3-l)(\Theta_{27}-c+14l)+10\big)-7(7l^3-21l^2-10l-1)$ which is in the ideal $\langle \Theta_{27}-c,7^{k}\rangle$ since we observed $7l^3-21l^2-10l-1\equiv0\Mod{7^{k-1}}$.

Now we prove $ \langle \Theta_{27}-c,7^k\rangle\subset \langle (3-l)(\Theta_{27}-2)^2+10(\Theta_{27}-2)+7,7^{k-2}(\Theta_{27}-2)^2\rangle$. We consider two equalities
\begin{align*}\tag{a} 7^k&=7^{k-1}\big((3-l)(\Theta_{27}-2)^2+10(\Theta_{27}-2)+7\big)-\big((3-l)(\Theta_{27}-2)+10\big)7^{k-1}(\Theta_{27}-2).\\
\tag{b}
7^{k-1}&=7^{k-2}\big((3-l)(\Theta_{27}-2)^2+10(\Theta_{27}-2)+7\big)-\big((3-l)(\Theta_{27}-2)+10\big)7^{k-2}(\Theta_{27}-2).
\end{align*}
From (a) and (b), $7^{k-1}(\Theta_{27}-2)$ and $7^k$ are in $ \langle (3-l)(\Theta_{27}-2)^2+10(\Theta_{27}-2)+7,7^{k-2}(\Theta_{27}-2)^2\rangle$.

It remains to prove $\Theta_{27}-c\in \langle (3-l)(\Theta_{27}-2)^2+10(\Theta_{27}-2)+7,7^{k-2}(\Theta_{27}-2)^2\rangle$. Since $7^{k-1}$ and $7(10l-1)(3-l)+100$ are coprime, it suffices to prove 
\begin{align*}
\tag{c} 7^{k-1}(\Theta_{27}-c)\in  \langle (3-l)(\Theta_{27}-2)^2+10(\Theta_{27}-2)+7,7^{k-2}(\Theta_{27}-2)^2\rangle.\\
\tag{d} (7(10l-1)(3-l)+100)(\Theta_{27}-c)\in  \langle (3-l)(\Theta_{27}-2)^2+10(\Theta_{27}-2)+7,7^{k-2}(\Theta_{27}-2)^2\rangle.
\end{align*}
By (b), (c) directly follows. Consider
\begin{align*}
&\big(7(10l-1)(3-l)+100\big)(\Theta_{27}-c)\\
&=\big(70l(3-l)+100-7(3-l)\big)(\Theta_{27}-c)\\
&=10\big((3-l)(\Theta_{27}-2)+7l(3-l)+10\big)(\Theta_{27}-c)-(3-l)\big(10(\Theta_{27}-2)+7\big)(\Theta_{27}-c)\\
&=10\big((3-l)(\Theta_{27}-c)+14l(3-l)+10\big)(\Theta_{27}-c)-(3-l)\big(10(\Theta_{27}-2)+7\big)(\Theta_{27}-c).
\end{align*}
Therefore, to prove (d), it suffices to prove that the following two terms in (\ref{item:twoideals}) are in the ideal $\langle (3-l)(\Theta_{27}-2)^2+10(\Theta_{27}-2)+7,7^{k-2}(\Theta_{27}-2)^2\rangle$:
\[\tag{e}\label{item:twoideals} \big((3-l)(\Theta_{27}-c)+14l(3-l)+10\big)(\Theta_{27}-c)\textrm{~and~} \big(10(\Theta_{27}-2)+7\big)(\Theta_{27}-c).\]
The first term $\big((3-l)(\Theta_{27}-c)+14l(3-l)+10\big)(\Theta_{27}-c)$ of (e) is equal to 
\begin{align*}
(3-l)(\Theta_{27}-2)^2+10(\Theta_{27}-2)+7-\big((3-l)49l^2+70l+7\big).
\end{align*}
Recall that we observed $7l^3-21l^2-10l-1\equiv 0\Mod{7^{k-1}}$ in the beginning of the proof. It follows that $(3-l)49l^2+70l+7=-7(7l^3-21l^2-10l-1)\equiv 0\Mod{7^k}$. Since we proved $7^k\in \langle (3-l)(\Theta_{27}-2)^2+10(\Theta_{27}-2)+7,7^{k-2}(\Theta_{27}-2)^2\rangle$, the term $\big((3-l)(\Theta_{27}-c)+14l(3-l)+10\big)(\Theta_{27}-c)$ is also in $\langle (3-l)(\Theta_{27}-2)^2+10(\Theta_{27}-2)+7,7^{k-2}(\Theta_{27}-2)^2\rangle$.

To show the second term $\big(10(\Theta_{27}-2)+7\big)(\Theta_{27}-c)$ of (e) is in $\langle (3-l)(\Theta_{27}-2)^2+10(\Theta_{27}-2)+7,7^{k-2}(\Theta_{27}-2)^2\rangle$, consider
\begin{align*}&\big(10(\Theta_{27}-2)+7\big)(\Theta_{27}-c)\\&=\big(10(\Theta_{27}-2)+7\big)(\Theta_{27}-2-7l)\\
&=10(\Theta_{27}-2)^2+7(\Theta_{27}-2)-70l(\Theta_{27}-2)-49l\\
&=10(\Theta_{27}-2)^2+7(\Theta_{27}-2)+(3-l)70(\Theta_{27}-2)+49(3-l)-210(\Theta_{27}-2)-147\\
&=10(\Theta_{27}-2)^2+7(\Theta_{27}-2)+(3-l)\big(21(\Theta_{27}-2)^2+70(\Theta_{27}-2)+49\big)-21(3-l)(\Theta_{27}-2)^2\\&-210(\Theta_{27}-2)-147\\
&=10(\Theta_{27}-2)^2+7(\Theta_{27}-2)+(3-l)(\Theta_{27}-2)^3-21\big((3-l)(\Theta_{27}-2)^2+10(\Theta_{27}-2)+7\big)\\
&=(\Theta_{27}-23)\big((3-l)(\Theta_{27}-2)^2+10(\Theta_{27}-2)+7\big).
\end{align*}
Note that we used the equality $(\Theta_{27}-2)^3=21(\Theta_{27}-2)^2+70(\Theta_{27}-2)+49$. This completes the proof.
	\end{proof}
\begin{theorem}\label{theorem:idealclassmonoid27}The ideal class monoid $C(\Z[\Theta_{27}])$ consists of the following $7$ elements, 
\begin{align*}I_0&=[\langle \Theta_{27}-2,7\rangle],\\
 I_1&=[\langle \Theta_{27}-7,17\rangle],\\ I_2&=[\langle\Theta_{27}-4,5\rangle],\\ I_3&=[\langle \Theta_{27}-11,13\rangle],\\
I_4&=[\langle \Theta_{27}-10,11\rangle],\\ I_5&=[\langle \Theta_{27}-14,19\rangle],\\ I_6&=[\langle \Theta_{27},-1,1\rangle].
 \end{align*}
 The multiplication table of $C(\Z[\Theta_{27}])$ is given in Table~\ref{table:multiplicationtable}.
\end{theorem}

\begin{table}[htb!]
\centering
\begin{tabular}{c|c|c|c|c|c|c|c}
      & $I_0$ & $I_1$ & $I_2$ & $I_3$ & $I_4$ & $I_5$ & $I_6$ \\ \hline
$I_0$ & $I_0$ & $I_0$ & $I_0$ & $I_0$ & $I_0$ & $I_0$ & $I_0$ \\ \hline
$I_1$ & $I_0$ & $I_2$ & $I_3$ & $I_4$ & $I_5$ & $I_6$ & $I_1$ \\ \hline
$I_2$ & $I_0$ & $I_3$ & $I_4$ & $I_5$ & $I_6$ & $I_1$ & $I_2$ \\ \hline
$I_3$ & $I_0$ & $I_4$ & $I_5$ & $I_6$ & $I_1$ & $I_2$ & $I_3$ \\ \hline
$I_4$ & $I_0$ & $I_5$ & $I_6$ & $I_1$ & $I_2$ & $I_3$ & $I_4$ \\ \hline
$I_5$ & $I_0$ & $I_6$ & $I_1$ & $I_2$ & $I_3$ & $I_4$ & $I_5$ \\ \hline
$I_6$ & $I_0$ & $I_1$ & $I_2$ & $I_3$  & $I_4$ & $I_5$ & $I_6$ \\ 
\end{tabular}
\caption{The multiplication table of $C(\Z[\Theta_{27}])$.}
\label{table:multiplicationtable}
\end{table}

\begin{proof}In Section~\ref{subsection:27MAGMA}, we observe that $\operatorname{Pic}(\Z[\Theta_{27}])$ consists of $I_1,I_2,\ldots,I_6$, and the multiplication is given by $I_i\cdot I_j=I_{i+j}$ for all $1\leq i,j\leq 6$ where subscripts are understood modulo $6$. It suffices to analyze equivalence classes of non-invertible ideals of $\Z[\Theta_{27}]$. By Proposition~\ref{proposition:27}, every non-zero, non-invertible ideal of $\Z[\Theta_{27}]$ is equivalent to $\langle\Theta_{27}-2,7\rangle\cdot J$ for some invertible ideal $J$. Let $I_0$ be the equivalence class of the non-invertible ideal $\langle \Theta_{27}-2,7\rangle$. 

Since $\operatorname{Pic}(\Z[\Theta_{27}])$ consists of $I_i$ for $i=1,\ldots,6$, the equivalence class of $J$ is $I_i$ for some $1\leq i\leq 6$. In Section~\ref{subsection:27MAGMA}, we observe the following.
\begin{enumerate}
\item For $i=1,\ldots,6$, each $I_i$ is represented by the ideal $\langle \Theta_{27}-2-7k,49\rangle$ for some $k=0,1,3,4,5,6$.
\item $\langle \Theta_{27}-2,49\rangle$ is a principal ideal. 
\item For any $k=1,3,4,5,6$, $\langle \Theta_{27}-2,7\rangle\cdot\langle\Theta_{27}-2,49\rangle=\langle \Theta_{27}-2,7\rangle\cdot\langle\Theta_{27}-2-7k,49\rangle$.
\end{enumerate}
Since $\langle \Theta_{27}-2,49\rangle$ is a principal ideal and $\langle\Theta_{27}-2,7\rangle$ represents $I_0$, these observations imply that 
\[I_0\cdot I_i=I_i\cdot I_0=I_0\]
for any $i=1,\ldots,6$. To obtain Table~\ref{table:multiplicationtable}, it remains to show that $I_0\cdot I_0=I_0$. For this, we show that 
\[\langle \Theta_{27}-2,7\rangle\cdot \langle \Theta_{27}-2,7\rangle\approx \langle \Theta_{27}-2,7\rangle.\]
Recall that $f_{27}(x)=x^3-27x^2+26x-1=(x-2)^3-7\big(3(x-2)^2+10(x-2)+7\big)$. It follows that $(\Theta_{27}-2)^3=7\big(3(\Theta_{27}-2)^2+10(\Theta_{27}-2)+7\big)$. Then we have 
\begin{align*}
\langle \Theta_{27}-2,7\rangle &\approx \langle (\Theta_{27}-2)^3,7(\Theta_{27}-2)^2\rangle \\
&=\langle 7\big(3(\Theta_{27}-2)^2+10(\Theta_{27}-2)+7\big), 7(\Theta_{27}-2)^2\rangle.
\end{align*}
On the other hand, we have
\begin{align*}\langle \Theta_{27}-2,7\rangle\cdot \langle \Theta_{27}-2,7\rangle&=\langle (\Theta_{27}-2)^2, 7(\Theta_{27}-2),49\rangle\\
&\approx \langle (\Theta_{27}-2)^3,7(\Theta_{27}-2)^2,49(\Theta_{27}-2)\rangle\\
&= \langle 7\big(3(\Theta_{27}-2)^2+10(\Theta_{27}-2)+7\big),7(\Theta_{27}-2)^2,49(\Theta_{27}-2)\rangle.
\end{align*}
Note that 
\begin{align*}49(\Theta_{27}-2)&=7\big(3(\Theta_{27}-2)^2+10(\Theta_{27}-2)+7\big)(\Theta_{27}-2)-7\big(3(\Theta_{27}-2)+10\big)(\Theta_{27}-2)^2\\
&\in \langle 7\big(3(\Theta_{27}-2)^2+10(\Theta_{27}-2)+7\big),7(\Theta_{27}-2)^2\rangle.\end{align*}
It follows that 
\[\langle \Theta_{27}-2,7\rangle\cdot \langle \Theta_{27}-2,7\rangle\approx \langle 7\big(3(\Theta_{27}-2)^2+10(\Theta_{27}-2)+7\big),7(\Theta_{27}-2)^2\rangle\approx \langle \Theta_{27}-2,7\rangle,\]
and this completes the proof.
\end{proof}

	\section{Finding representatives of elements in $\operatorname{Pic}(\Z[\Theta_{n}])$}\label{section:findingrepresentatives}
	In this section, we use MAGMA to find representatives of elements of $\operatorname{Pic}(\Z[\Theta_n])$. 
	
	\begin{definition}Let $x$ be an element of $C(\Z[\Theta_n])$. We say a tuple $(c,d,n)\in \mathcal{CS}$ is a \emph{representative} of $x$ if the integers $c$ and $d$ satisfy $1\leq c\leq d$ and $x=[\langle \Theta_n-c,d\rangle]$. (Recall that $(c,d,n)\in \mathcal{CS}$ if and only if $f_n(c)\equiv 0 \Mod{d}$.) We say a representative $(c,d,n)$ of $x$ is \emph{minimal} if $(c',d',n)$ is another representative of $x$, then either $d'>d$ or $d'=d$ and $c'>c$.
	\end{definition} 
	\begin{remark}Since $\langle \Theta_n-1,1\rangle$ is principal, the minimal representative of the trivial element in $C(\Z[\Theta_n])$ is $(1,1,n)$.  Every element $x$ of $C(\Z[\Theta_n])$ has a representative by Proposition~\ref{proposition:Aitchison-Rubinstein}. The minimal representative of $x$ is  minimal with respect to the colexicographic order on the set of representatives of $x$. Each tuple $(c,d,n)$ corresponds to the standard \CS matrix $X_{c,d,n}$.
			\end{remark}

	\subsection{The cases that $\Z[\Theta_n]$ is a Dedekind domain}
	In this subsection, we assume that $\Z[\Theta_n]$ is a Dedekind domain, that is, $\operatorname{Pic}(\Z[\Theta_n])=C(\Z[\Theta_n])$ is a group. We give two pseudocodes
 each of which computes the following:
  \begin{enumerate}
 \item The list of minimal representatives $(c,d,n)\in \mathcal{CS}$ such that $d\leq N$ for any given integers $N>0$ and $n$ such that $\Z[\Theta_n]$ is a Dedekind domain. 
\item The representatives $(c,d,n)$ of $x$ such that $d\leq N$ for a given representative $(c_0,d_0,n)$ of an element $x$ of $C(\Z[\Theta_n])$ and a given integer $N>0$. (Equivalently, for a given standard \CS matrix $X_{c_0,d_0,n}$ and an integer $N>0$, the pseudocode computes the set of standard \CS matrices $X_{c,d,n}$ such that $d\leq N$.)
\end{enumerate}

\begin{algorithm}[htb!]
\caption{Finding minimal representatives of $C(\Z[\Theta_n])$ when $\Z[\Theta_n]$ is a Dedekind domain}
\label{algorithm:1}
\begin{algorithmic}[1]
\State $i=1;$
\While{$ i<\# C(\Z[\Theta_n])$}
\For{$1\leq d<N$}
\For{$1\leq c\leq d$}
\If{$f_n(c)\equiv 0\Mod d$ and $\langle \Theta_n-c,d\rangle$ is not a principal ideal}
\If{$i>1$ and $[\langle\Theta_n-c_i,d_i\rangle]\neq [\langle \Theta_n-c_j,d_j\rangle]$ for any $1\leq j<i$}\State let $(c_i,d_i)=(c,d)$ and $i=i+1$;
\EndIf

\If{$i=1$}
\State let $(c_i,d_i)=(c,d)$ and $i=i+1$;
\EndIf
\EndIf
\EndFor
\EndFor
\EndWhile
\State print $(1,1,n)$;
\For{$1\leq i<\# C(\Z[\Theta_n])$} print $(c_i,d_i,n)$;
\EndFor
\end{algorithmic}
\end{algorithm} 
\begin{algorithm}[htb!]
\caption{Finding other representatives of $x$ when $\Z[\Theta_n]$ is a Dedekind domain}
\label{algorithm:2}
\begin{algorithmic}[1]
\For{$1\leq d<N$}
\For{$1\leq c\leq d$}
\If{$f_n(c)\equiv 0\Mod {d}$ and $[\langle\Theta_n-c,d\rangle]=[\langle \Theta_n-c_0,d_0\rangle]$}
 print $(c,d,n)$;
\EndIf
\EndFor
\EndFor
\end{algorithmic}
\end{algorithm}

We give the corresponding MAGMA codes in Section~\ref{subsection:MAGMAcodes}, and these MAGMA codes will be used in the proof of Theorem~\ref{theorem:B} given in Section~\ref{section:theoremB}.

\subsection{MAGMA codes}\label{subsection:MAGMAcodes}
	In this subsection, we give MAGMA codes. One can execute the codes by pasting them to the online MAGMA calculator \url{http://magma.maths.usyd.edu.au/calc/}.
We first give a MAGMA code for Algorithm~\ref{algorithm:1}.
	\begin{Verbatim}
n := 69; 
N := 400; 
R<x> := PolynomialRing(Integers());
K<Theta> := NumberField(x^3-n*x^2+(n-1)*x-1);
O := EquationOrder(K);
f := x^3-n*x^2+(n-1)*x-1;
C := Order(RingClassGroup(O));       
X := ZeroMatrix(IntegerRing(), 2, C); 
i := 1;
for d in [1 .. N] do
	for c in [1 .. d] do
		if i eq C then break; end if;
		k := Evaluate(f, c);
		I := ideal< O | Theta-c, d >;
		if IsDivisibleBy(k, d) eq true and IsPrincipal(I) ne true then 
			if i ne 1 then
				IsSame := false;
				for j in [1 .. (i-1)] do 
					if  ClassRepresentative(ideal< O | Theta-c, d >) eq 
					ClassRepresentative(ideal< O | Theta-X[1][j], X[2][j]>) then 
						IsSame := true;  break; 
					end if;
				end for;
				if IsSame eq false then 
					X[1][i] := c; X[2][i] := d; i+:=1;
				end if;
			end if;
			if i eq 1 then
				X[1][i] := c; X[2][i] := d; i+:=1;
			end if;
		end if;
	end for;
end for; 
"There are", C, "similarity classes of trace",n,"Cappell-Shaneson matrices.";
print [1,1,n];
for i in [1 .. C-1] 
	do print [X[1][i],X[2][i],n];
end for;
\end{Verbatim}
Now we give a MAGMA code for Algorithm~\ref{algorithm:2}. 
\begin{Verbatim}
n := 70; 
c0 := 110; 
d0 := 189; 
N := 300;    
R<x> := PolynomialRing(Integers());
K<theta> := NumberField(x^3-n*x^2+(n-1)*x-1);
O := EquationOrder(K);
f := x^3-n*x^2+(n-1)*x-1;
C := Order(RingClassGroup(O));  
for d in [1.. N] do
	for c in [1 .. d] do
		if IsDivisibleBy(Evaluate(f,c),d) eq true and 
		ClassRepresentative(ideal< O | theta-c,d >) eq 
		ClassRepresentative(ideal< O | theta-c0,d0>) eq true 
		then [c,d,n]; 
		end if;  
	end for;
end for;
\end{Verbatim}
\subsection{Representatives of elements of $\operatorname{Pic}(\Z[\Theta_{27}])$}\label{subsection:27MAGMA}
From the MAGMA code given below, $\operatorname{Pic}(\Z[\Theta_{27}])\cong \Z_6$ with a generator is represented by the ideal 
\[I=\langle 1+18601\Theta_{27}^2,\Theta_{27}+3672\Theta_{27}^2,26737\Theta_{27}^2\rangle.\]
We check that $I^5\cdot \langle \Theta_{27}-7,17\rangle$ and $\langle \Theta_{27}-23,49\rangle\cdot\langle \Theta_{27}-7,17\rangle$
are principal ideals. Since $I^6$ is a principal ideal, we can conclude that $I\approx \langle \Theta_{27}-7,17\rangle$ and $I^5\approx \langle \Theta_{27}-23,49\rangle$. Similarly, we have
\begin{align*}
I&\approx\langle \Theta_{27}-7,17\rangle\approx \langle \Theta_{27}-44,49\rangle,\\
 I^2&\approx \langle\Theta_{27}-4,5\rangle\approx \langle \Theta_{27}-9, 49\rangle,\\
I^3&\approx \langle \Theta_{27}-11,13\rangle\approx \langle \Theta_{27}-30, 49\rangle,\\
I^4&\approx \langle \Theta_{27}-10,11\rangle\approx \langle \Theta_{27}-37,49\rangle,\\
I^5&\approx \langle \Theta_{27}-14,19\rangle\approx \langle \Theta_{27}-23,49\rangle,\\
I^6&\approx \langle \Theta_{27}-1,1\rangle\approx \langle \Theta_{27}-2,49\rangle.
\end{align*}
Consequently, the elements of $\operatorname{Pic}(\Z[\Theta_{27}])\cong \Z_6$ have representatives 
\[(1,1,27), (4,5,27), (10,11,27), (11,13,27), (7,17,27), (14,19,27).\]
 Note that these representatives are actually in $\mathcal{CS}$ since 
\begin{align*}
f_{27}(1)&=1^3-27\cdot 1^2+26\cdot 1-1=-1\equiv 0\Mod{1},\\
f_{27}(4)&=4^3-27\cdot 4^2+26\cdot 4-1=-265\equiv 0\Mod{5},\\
f_{27}(10)&=10^3-27\cdot 10^2+26\cdot 10-1=-1441\equiv 0\Mod{11},\\
f_{27}(11)&=11^3-27\cdot 11^2+26\cdot 11-1=-1651\equiv 0\Mod {13},\\
f_{27}(7)&=7^3-27\cdot 7^2+26\cdot 7-1=-799\equiv 0\Mod {17},\\
f_{27}(14)&=14^3-27\cdot 14^2+26\cdot 14-1=-2185\equiv 0\Mod{19}.
\end{align*}
To prove Theorem~\ref{theorem:idealclassmonoid27}, for $k=1,3,4,5,6$, we also observe that
\[\langle \Theta_{27}-2,7\rangle\cdot \langle \Theta_{27}-2,49\rangle=\langle \Theta_{27}-2,7\rangle\cdot \langle \Theta_{27}-2-7k,49\rangle\]
Here is the MAGMA code used in above.
\begin{Verbatim}
R<x> := PolynomialRing(Integers());
f := x^3-27*x^2+26*x-1;
K<theta> := NumberField(f);
O := EquationOrder(K);
C,g := RingClassGroup(O);
I := g(C.1);
C; 
g; 
I;
IsPrincipal(I^5*ideal<O | theta-7,17>);
IsPrincipal(ideal<O | theta-23,49>*ideal<O | theta-7,17>);
IsPrincipal(I^4*ideal<O | theta-4,5>);
IsPrincipal(ideal<O | theta-37,49>*ideal<O | theta-4,5>);
IsPrincipal(I^3*ideal<O | theta-11,13>);
IsPrincipal(ideal<O | theta-30,49>*ideal<O | theta-11,13>);
IsPrincipal(I^2*ideal<O | theta-10,11>);
IsPrincipal(ideal<O | theta-9,49>*ideal<O | theta-10,11>);
IsPrincipal(I*ideal<O | theta-14,19>);
IsPrincipal(ideal<O | theta-44,49>*ideal<O | theta-14,19>);
IsPrincipal(ideal<O | theta-2,49>);
ideal<O | theta-2,7>*ideal<O | theta-2,49>;
ideal<O | theta-2,7>*ideal<O | theta-9,49>;
ideal<O | theta-2,7>*ideal<O | theta-23,49>;
ideal<O | theta-2,7>*ideal<O | theta-30,49>;
ideal<O | theta-2,7>*ideal<O | theta-37,49>;
ideal<O | theta-2,7>*ideal<O | theta-44,49>;
\end{Verbatim}	

\clearpage
\subsection{Representatives of elements of $C(\Z[\Theta_n])$}
\begin{table}[htb!]
\begin{center}
\begin{tabular}{ccl}
$n$ & $\# C(\Z[\Theta_n])$ & Representatives of elements of $C(\Z[\Theta_n])$ \\ \hline
3 & 1 & (1,1,3)\\
4 & 1 & (1,1,4)\\
5 & 1 & (1,1,5)\\
6 & 1 & (1,1,6)\\
7 & 1 & (1,1,7)\\
8 & 1 & (1,1,8)\\
9 & 1 & (1,1,9)\\
10 & 2 & (1,1,10), (2,3,10) \\
11 & 1 & (1,1,11)\\
12 & 2 & (1,1,12), (4,5,12) \\
13 & 3 & (1,1,13), (2,3,13), (3,5,13) \\
14 & 2 & (1,1,14), (5,7,14)  \\
15 & 2 & (1,1,15), (2,5,15)  \\
16 & 3 & (1,1,16), (2,3,16), (4,7,16)  \\
17 & 3 & (1,1,17), (4,5,17), (6,7,17) \\
18& 2& (1,1,18), (3,5,18)  \\
19& 6& (1,1,19), (2,3,19), (3,7,19), (5,9,19), (8,11,19),\\
&& (9,11,19)\\
20& 3& (1,1,20), (2,5,20), (2,7,20)\\
21&3& (1,1,21), (5,7,21), (9,13,21)\\
22&6& (1,1,22), (2,3,22), (4,5,22), (8,9,22), (6,11,22),\\
&& (14,17,22)\\
23 & 5&(1,1,23), (3,5,23), (4,7,23), (5,11,23), (6,13,23)\\
24 & 4&(1,1,24), (6,7,24), (3,11,24), (17,23,24)\\
25 & 9 &(1,1,25), (2,3,25), (2,5,25), (2,9,25), (7,11,25),\\
&& (7,13,25), (8,13,25), (10,13,25), (13,17,25)\\
26&4&(1,1,26), (3,7,26), (4,11,26), (12,17,26)\\
27&7&(1,1,27), (4,5,27), (2,7,27), (10,11,27), (11,13,27),\\
&& (7,17,27), (14,19,27)\\
28&10&(1,1,28), (2,3,28), (3,5,28), (5,7,28), (5,9,28),\\
&& (8,15,28), (13,19,28), (16,19,28), (19,23,28), (23,27,28)\\
29&4&(1,1,29), (4,17,29), (8,19,29), (27,37,29)\\
30&8&(1,1,30), (2,5,30), (4,7,30), (2,11,30), (8,11,30),\\
&& (9,11,30), (5,13,30), (15,17,30)\\
31&7&(1,1,31), (2,3,31), (6,7,31), (8,9,31), (9,17,31),\\
&& (11,17,31), (15,23,31)\\
32&6&(1,1,32), (4,5,32), (3,13,32), (4,13,32), (12,13,32),\\
&&(18,23,32)\\
33&7&(1,1,33), (3,5,33), (3,7,33), (6,11,33), (12,19,33),\\
&&(16,23,33), (35,43,33)\\

34&12&(1,1,34), (2,3,34), (2,7,34), (2,9,34), (5,11,34),\\
 &&(9,13,34),(5,17,34), (9,19,34), (10,19,34), (15,19,34),\\
 && (20,27,34), (26,41,34)\\
35&10&(1,1,35), (2,5,35), (5,7,35), (3,11,35), (2,13,35),\\
&& (3,17,35), (4,19,35), (17,25,35), (13,29,35), (17,37,35)\\
36&5&(1,1,36), (7,11,36), (6,13,36), (8,17,36), (11,19,36)\\

\end{tabular}
\caption{Representatives of elements of $C(\Z[\Theta_n])$ for $3\leq n\leq 36$.}
\label{table:3<=n<=36}
\end{center}
\end{table}

\begin{table}
\begin{center}
\begin{tabular}{ccl}

$n$ & $\# C(\Z[\Theta_n])$ & Representatives of elements of $C(\Z[\Theta_n])$ \\ \hline
37&15&(1,1,37), (2,3,37), (4,5,37), (4,7,37), (5,9,37),\\
&& (4,11,37), (14,15,37), (6,17,37), (11,21,37), (5,23,37),\\
&& (7,23,37), (5,27,37), (26,33,37), (14,45,37), (23,51,37) \\
38 & 12&(1,1,38), (3,5,38), (6,7,38), (10,11,38), (7,13,38),\\
&& (8,13,38), (10,13,38), (10,17,38), (6,19,38), (8,25,38),\\
&& (22,29,38), (43,55,38)\\
39&6&(1,1,39), (14,17,39), (17,29,39), (25,29,39), (26,29,39),\\
&& (21,37,39)\\
40&16&(1,1,40), (2,3,40), (2,5,40), (3,7,40), (8,9,40),\\
&& (11,13,40), (2,15,40), (17,19,40), (17,21,40), (12,23,40),\\
&& (20,29,40), (12,31,40), (25,31,40), (14,37,40), (30,41,40),\\
&& (17,57,40)\\
41&9&(1,1,41), (2,7,41), (2,11,41), (8,11,41), (9,11,41),\\
&& (7,19,41), (21,23,41), (9,29,41), (28,43,41)\\
42&10&(1,1,42), (4,5,42), (5,7,42), (13,17,42), (16,17,42),\\
&& (20,23,42), (19,25,42), (21,31,42), (19,43,42), (31,71,42)\\
43&16&(1,1,43), (2,3,43), (3,5,43), (2,9,43), (5,13,43),\\
&& (8,15,43), (12,17,43), (5,19,43), (6,23,43), (13,25,43),\\
&& (10,43,43), (38,45,43), (29,51,43), (15,53,43), (54,61,43),\\
&& (10,67,43)\\
44&9&(1,1,44), (4,7,44), (6,11,44), (7,17,44), (23,29,44),\\
&& (9,31,44), (13,31,44), (22,31,44), (8,37,44)\\
45&14&(1,1,45), (2,5,45), (6,7,45), (5,11,45), (3,13,45),\\
&& (4,13,45), (12,13,45), (2,17,45), (3,19,45), (12,25,45),\\
&& (27,31,45), (27,35,45), (42,53,45), (24,61,45)\\
46&12&(1,1,46), (2,3,46), (5,9,46), (3,11,46), (4,17,46),\\
&&(14,19,46), (3,23,46), (10,23,46), (14,27,46), (4,29,46),\\
&& (11,29,46), (14,33,46)\\
47&16&(1,1,47), (4,5,47), (3,7,47), (7,11,47), (9,13,47),\\
&& (15,17,47), (13,19,47), (16,19,47), (18,19,47), (17,23,47),\\
&& (14,25,47), (28,31,47), (24,35,47), (18,41,47), (32,43,47),\\
&& (39,83,47)\\
48&18&(1,1,48), (3,5,48), (2,7,48), (4,11,48), (2,13,48),\\
&& (9,17,48), (11,17,48), (8,19,48), (8,23,48), (9,23,48),\\
&& (18,25,48), (5,29,48), (20,31,48), (23,35,48), (25,43,48),\\
&& (15,47,48), (54,67,48), (39,71,48)\\
49&20&(1,1,49), (2,3,49), (5,7,49), (8,9,49), (10,11,49),\\
&& (6,13,49), (5,21,49), (4,23,49), (11,23,49), (8,27,49),\\
&& (15,29,49), (25,37,49), (29,37,49), (32,37,49), (32,39,49),\\
&& (13,43,49), (38,43,49), (50,69,49), (33,73,49), (41,89,49)\\
50&12&(1,1,50), (2,5,50), (2,19,50), (14,23,50), (22,25,50),\\
&&(19,31,50), (11,37,50), (15,37,50), (24,37,50), (13,41,50),\\
&& (9,43,50), (48,61,50) \\

\end{tabular}
\caption{Representatives of elements of $C(\Z[\Theta_n])$ for $37\leq n\leq 50$.}
\label{table:37<=n<=50}
\end{center}
\end{table}

\clearpage
\begin{table}[htb]
\begin{center}
\begin{tabular}{ccl}
$n$ & $\# C(\Z[\Theta_n])$ & Representatives of elements of $C(\Z[\Theta_n])$ \\ \hline

51&13&(1,1,51), (4,7,51), (7,13,51), (8,13,51), (10,13,51),\\
&& (5,17,51), (13,23,51), (19,23,51), (12,29,51), (24,31,51),\\
&&(31,41,51), (38,47,51), (43,47,51)\\
52 & 28& (1,1,52), (2,3,52), (4,5,52), (6,7,52), (2,9,52),\\
&&(2,11,52), (8,11,52), (9,11,52), (14,15,52), (3,17,52),\\
&&(12,19,52), (20,21,52), (9,25,52), (11,27,52), (27,29,52),\\
&&(16,31,52), (8,33,52), (20,33,52), (36,41,52), (29,45,52),\\
&& (20,51,52), (24,55,52), (20,63,52), (41,71,52), (59,73,52),\\
&& (59,75,52), (56,87,52), (32,103,52)\\
53 &15 & (1,1,53), (3,5,53), (11,13,53), (8,17,53), (9,19,53),\\
&& (10,19,53), (15,19,53), (23,25,53), (8,29,53), (21,29,53),\\
&& (24,29,53), (7,31,53), (24,41,53), (37,53,53), (45,83,53)\\
54&12 & (1,1,54), (3,7,54), (6,17,54), (4,19,54), (15,23,54),\\
&& (19,29,54), (4,31,54), (5,31,54), (14,31,54), (31,37,54),\\
&& (28,41,54), (25,53,54)\\
55&27& (1,1,55), (2,3,55), (2,5,55), (2,7,55), (5,9,55),\\
&& (6,11,55), (2,15,55), (10,17,55), (11,19,55), (2,21,55),\\
&& (18,23,55), (7,25,55), (23,27,55), (17,33,55), (39,43,55),\\
&& (32,45,55), (39,47,55), (44,51,55), (44,53,55), (17,55,55),\\
&& (11,57,55), (17,61,55), (29,67,55), (41,69,55), (20,73,55),\\
&& (32,75,55), (27,85,55)\\

56&15&(1,1,56), (5,7,56), (5,11,56), (5,13,56), (14,17,56),\\
&& (16,23,56), (15,31,56), (16,37,56), (38,41,56), (29,43,56),\\
&& (11,47,56), (20,47,56), (73,89,56), (75,101,56), (78,107,56)\\
57&16&(1,1,57), (4,5,57), (3,11,57), (6,19,57), (22,23,57),\\
&& (4,25,57), (3,29,57), (7,29,57), (18,29,57), (12,37,57),\\
&& (10,41,57), (43,53,57), (41,73,57), (15,79,57), (25,89,57),\\
&& (15,109,57)\\
58&36&(1,1,58), (2,3,58), (3,5,58), (4,7,58), (8,9,58),\\
&& (7,11,58), (3,13,58), (4,13,58), (12,13,58), (8,15,58),\\
&& (11,21,58), (3,25,58), (17,27,58), (17,31,58), (18,31,58),\\
&& (23,31,58), (29,33,58), (18,35,58), (23,37,58), (33,37,58),\\
&& (17,39,58), (29,39,58), (19,41,58), (20,43,58), (8,45,58),\\
&& (36,47,58), (29,53,58), (8,61,58), (53,63,58), (53,75,58),\\
&& (56,79,58), (25,91,58), (20,109,58), (107,117,58), (101,123,58),\\
&& (83,141,58)\\
59&14&(1,1,59), (6,7,59), (4,11,59), (13,17,59), (16,17,59),\\
&& (17,19,59), (10,29,59), (26,31,59), (28,37,59), (41,49,59),\\
&&(42,59,59), (26,61,59), (38,61,59), (55,67,59)\\
60&16&(1,1,60), (2,5,60), (10,11,60), (9,13,60), (12,17,60),\\
&& (7,19,60), (2,23,60), (5,23,60), (7,23,60), (17,25,60),\\
&& (6,37,60), (6,43,60), (5,47,60), (41,53,60), (32,55,60),\\
&& (48,83,60)\\
\end{tabular}
\caption{Representatives of elements of $C(\Z[\Theta_n])$ for $51\leq n\leq 60$.}
\label{table:51<=n<=60}
\end{center}
\end{table}

\clearpage

\begin{table}[htb]
\begin{center}
\begin{tabular}{ccl}
$n$ & $\# C(\Z[\Theta_n])$ & Representatives of elements of $C(\Z[\Theta_n])$ \\ \hline

61&21&(1,1,61), (2,3,61), (3,7,61), (2,9,61), (2,13,61),\\
&& (7,17,61), (17,21,61), (20,27,61), (19,37,61), (30,43,61),\\
&& (32,47,61), (35,47,61), (41,47,61), (41,51,61), (51,59,61),\\
&& (53,59,61), (46,73,61), (23,79,61), (80,91,61), (85,103,61),\\
&& (26,139,61)\\
62&18&(1,1,62), (4,5,62), (2,7,62), (6,13,62), (2,17,62),\\
&& (5,19,62), (24,25,62), (14,29,62), (9,35,62), (22,37,62),\\
&& (23,41,62), (24,43,62), (21,53,62), (14,59,62), (28,61,62),\\
&& (19,65,62), (55,71,62), (63,73,62)\\
63&24&(1,1,63), (3,5,63), (5,7,63), (2,11,63), (8,11,63),\\
&& (9,11,63), (4,17,63), (12,23,63), (8,25,63), (10,31,63),\\
&& (11,31,63), (33,35,63), (4,41,63), (6,41,63), (12,41,63),\\
&& (33,49,63), (13,55,63), (60,73,63), (68,77,63), (74,83,63),\\
&& (28,97,63), (60,97,63), (72,97,63), (79,107,63)\\
64&30&(1,1,64), (2,3,64), (5,9,64), (7,13,64), (8,13,64),\\
&& (10,13,64), (15,17,64), (3,19,64), (21,23,64), (5,27,64),\\
&& (6,29,64), (13,29,64), (16,29,64), (8,39,64), (20,39,64),\\
&& (23,39,64), (5,43,64), (17,43,64), (18,47,64), (28,47,64),\\
&& (32,51,64), (40,53,64), (41,57,64), (37,59,64), (32,67,64),\\
&& (51,71,64), (74,87,64), (62,97,64), (75,109,64), (146,159,64)\\
65&21&(1,1,65), (2,5,65), (4,7,65), (9,17,65), (11,17,65),\\
&& (14,19,65), (20,23,65), (2,25,65), (32,35,65), (15,41,65),\\
&& (16,41,65), (34,41,65), (31,47,65), (11,49,65), (53,61,65),\\
&& (39,67,65), (52,67,65), (61,79,65), (58,83,65), (62,89,65),\\
&& (90,113,65)\\
66&20&(1,1,66), (6,7,66), (6,11,66), (11,13,66), (13,19,66),\\
&& (16,19,66), (18,19,66), (6,23,66), (8,31,66), (29,31,66),\\
&& (9,37,66), (27,37,66), (30,37,66), (29,41,66), (40,43,66),\\
&& (33,47,66), (20,49,66), (23,53,66), (68,79,66), (64,109,66)\\
67&28&(1,1,67), (2,3,67), (4,5,67), (8,9,67), (5,11,67),\\
&& (14,15,67), (8,19,67), (19,25,67), (26,27,67), (22,29,67),\\
&& (5,33,67), (5,37,67), (7,37,67), (18,37,67), (21,41,67),\\
&& (31,43,67), (34,43,67), (40,47,67), (39,53,67), (49,55,67),\\
&& (44,75,67), (26,81,67), (78,97,67), (71,99,67), (44,111,67),\\
&& (92,111,67), (41,173,67), (50,179,67)\\
68&24&(1,1,68), (3,5,68), (3,7,68), (3,11,68), (5,17,68),\\
&& (13,25,68), (17,29,68), (25,29,68), (26,29,68), (3,35,68),\\
&& (33,43,68), (45,49,68), (12,53,68), (20,53,68), (36,53,68),\\
&& (23,61,68), (34,61,68), (15,67,68), (36,67,68), (23,73,68),\\
&& (25,79,68), (37,89,68), (80,97,68), (126,197,68)\\
69&18&(1,1,69), (2,7,69), (7,11,69), (5,13,69), (3,17,69),\\
&& (2,19,69), (3,23,69), (10,23,69), (20,29,69), (20,37,69),\\
&& (22,53,69), (36,61,69), (49,67,69), (32,71,69), (57,73,69),\\
&& (24,107,69), (60,127,69), (80,181,69)\\

\end{tabular}

\caption{Representatives of elements of $C(\Z[\Theta_n])$ for $61\leq n\leq 69$.}
\label{table:61<=n<=69}
\end{center}
\end{table}
	\section{Even more Cappell-Shaneson spheres are standard}\label{section:theoremB}
	The goal of this section is to prove Theorem~\ref{theorem:B} and Corollary~\ref{corollary:D}. 
	\subsection{Proof of Theorem~\ref{theorem:B}}\label{subsection:proofofTheoremB}
	The statement of Theorem~\ref{theorem:B} is that Conjecture~\ref{conjecture:Gompf} is true for trace $n$ if $n$ is an integer such that $-64\leq n\leq 69$. By Theorem~\ref{theorem:A}, it suffices to check that Conjecture~\ref{conjecture:Gompf} is true for trace $n$ where $3\leq n\leq 69$. We will use the reformulation of Conjecture~\ref{conjecture:Gompf} and the notations given in Section~\ref{subsection:Gompfequivalence}. To simplify the proof, we first prove a lemma.
	
	\begin{lemma}\label{lemma:induction}Let $n> 3$ be an integer. Suppose that Conjecture~\ref{conjecture:Gompf} is true for trace $m$ if $3\leq m\leq n-1$. If every element of $C(\Z[\Theta_n])$ has a representative $(c,d,n)$ such that $n\equiv n_0\Mod {d}$ for some $6-n\leq n_0\leq n-1$, then Conjecture~\ref{conjecture:Gompf} is true for trace $n$.
	\end{lemma}
	\begin{proof}By Theorem~\ref{theorem:A} and the hypothesis, Conjecture~\ref{conjecture:Gompf} is true for trace $m$ if $6-n\leq m\leq n-1$. Let $x$ be an element of $C(\Z[\Theta_n])$, and $(c,d,n)$ be a representative of $x$ satisfying that $n= n_0+kd$ for some $6-n\leq n_0\leq n-1$ and $k\in \Z$. Then we have $(c,d,n)\sim_G (c,d,n_0)$ because $n=n_0+kd$. Since Conjecture~\ref{conjecture:Gompf} is true for trace $n_0$, $(c,d,n_0)\sim (1,1,2)$. It follows that 
	\[(c,d,n)\sim_G (c,d,n_0)\sim (1,1,2).\]
Therefore Conjecture~\ref{conjecture:Gompf} is true for trace $n$.
		\end{proof}
	\begin{proof}[Proof of Theorem~\ref{theorem:B}] In Table~\ref{table:3<=n<=36}, we give minimal representatives of non-trivial elements of $C(\Z[\Theta_n])$ for $3\leq n\leq 36$. In particular,  $C(\Z[\Theta_n])$ is trivial if $3\leq n\leq 9$ or $n=11$, and hence Conjecture~\ref{conjecture:Gompf} is true for trace $3\leq n\leq 9$, and for trace~11. Since $10\equiv 7\Mod 3$, by applying Lemma~\ref{lemma:induction} for $n=10$, we can see that Conjecture~\ref{conjecture:Gompf} is true for trace 10. similarly, $12\equiv 7\Mod 5$, by applying Lemma~\ref{lemma:induction} for $n=12$, we can see that Conjecture~\ref{conjecture:Gompf} is true for trace 12. 
	
	We can continue this argument to conclude that Conjecture~\ref{conjecture:Gompf} is true for trace $n$ for $3\leq n\leq 51$. In fact, by Lemma~\ref{lemma:induction}, it suffices to observe the following statement using Tables~\ref{table:3<=n<=36}--\ref{table:51<=n<=60}. For any $13\leq n\leq 51$, every non-trivial element of $C(\Z[\Theta_n])$ has minimal representative $(c,d,n)$ such that $n\equiv n_0$ for some $6-n\leq n_0\leq n-1$. 
	
	When $n=52$, the inductive argument works for all minimal representatives except $(32,103,52)$. We give a sequence of Gompf equivalences from $(32,103,52)$ to $(87,101,50)$:
	\[(32,103,52)\sim_G (32,103,-51)\sim_S (87,101,-51)\sim_G (87,101,50).\]
	Since Conjecture~\ref{conjecture:Gompf} is true for trace $50$, $(87,101,50)$ is also Gompf equivalent to $(1,1,2)$, and hence Conjecture~\ref{conjecture:Gompf} is true for trace $52$.
	
	As in the above, one can easily check that for any $53\leq n\leq 55$, every non-trivial element of $C(\Z[\Theta_n])$ has minimal representative $(c,d,n)$ such that $n\equiv n_0$ for some $6-n\leq n_0\leq n-1$ using Table~\ref{table:51<=n<=60}. Conjecture~\ref{conjecture:Gompf} is true for trace $53\leq n\leq 55$. For $56\leq n\leq 69$, we can similarly continue the inductive argument except few cases. For brevity of our discussion, we just record Gompf equivalences for these exceptional cases. 
		
		\begin{itemize}
		\item $(15,109,57)\sim_G (15,109,-52)\sim_S(18,79,-52)\sim_G (18,79,27)$.
	\item $(107,117,58)\sim_G (107,117,-59)\sim_S (29,109,-59)\sim_G (29,109,50)$.
	\item $(101,123,58)\sim_G (101,123,-65)\sim_S(30,47,-65)\sim_G (30,47,18)$.
	\item $(83,141,58)\sim_S(128,165,58)\sim_G (128,165,-107)\sim_S(38,119,-107)\sim_G\\
	(38,119,12)$.
	\item $(26,139,61)\sim_S(119,291,61)\sim_G(119,291,-230)\sim_S(302,391,-230)\sim_G\\
	(302,391,161)\sim_S(114,149,161)\sim_G(114,149,-15)$.
	\item $(146,159,64)\sim_G(146,159,-95)\sim_S(26,89,-95)\sim_G(26,89,-6)$.
	\item $(41,173,67)\sim_G(41,173,-106)\sim_S(210,233,-106)\sim_G(210,233,127)\sim_S\\
	(158,267,127)\sim_G(158,267,-140)\sim_S(153,179,-140)\sim_G(153,179,39)$.
	\item $(50,179,67)\sim_S(272,291,67)\sim_G(272,291,-224)\sim_S(142,397,-224)\sim_G\\
	(142,397,173)\sim_S(14,149,173)\sim_G(14,149,24)$.
	\item $(126,197,68)\sim_S(248,265,68)\sim_G(248,265,-197)\sim_S(170,407,-197)\sim_G\\
	(170,407,210)\sim_S(18,277,210)\sim_G(18,277,-67)\sim_S(38,205,-67)\sim_G\\
	(38,205,138)\sim_S(139,227,138)\sim_G(139,227,-89)\sim_S(70,97,-89)\sim_G\\
	(70,97,8)$.
	\item $(80,181,69)\sim_S(167,211,69)\sim_G(167,211,-142)\sim_S(218,269,-142)\sim_G\\
	(218,269,127)\sim_S(36,151,127)\sim_G(36,151,-24)$.
	\end{itemize}
	This completes the proof.
	\end{proof}
\subsection{Proof of Corollary~\ref{corollary:D}}
We show Corollary~\ref{corollary:D} which says that $\Sigma_{M_k}^\epsilon$ is diffeomorphic to $S^4$, but $M_k$ is not similar to $A_n$ for any integers $k$ and $n$ where 
\[M_k=\begin{bmatrix}0&14k+7&49k+24\\0&2&7\\1&0&49k+25
	\end{bmatrix}.\] 

\begin{proof}[Proof of Corollary~\ref{corollary:D}] Note that $M_k=X_{2,7,49k+27}$. As we mentioned in the introduction, $\Sigma_{M_k}^\epsilon$ is diffeomorphic to $S^4$ for any integer $k$ and $\epsilon\in \Z_2$ by Corollary~\ref{corollary:C} or its weaker version given in \cite[Theorem~3.2]{Gompf:2010-1}. It remains to show that $M_k$ is not similar to $A_n$ for any integers $k$ and $n$. By Proposition~\ref{proposition:Aitchison-Rubinstein}, the similarity class of $M_k$ corresponds to the ideal class $[\langle \Theta_{49k+27}-2,7\rangle]\in C(\Z[\Theta_{49k+27}])$. We proved in Proposition~\ref{proposition:49k+27} that the ideal $\langle \Theta_{49k+27}-2,7\rangle$ is not invertible, and hence represents a non-trivial element in $C(\Z[\Theta_{49k+27}])$. As we discussed in Remark~\ref{remark:identity}, the similarity class of $A_n$ corresponds to the trivial element in $C(\Z[\Theta_{n+2}])$. It follows that $M_k$ is not similar to $A_n$ for any $k$ and $n$.
\end{proof}
\section{A note on Earle's result on \CS matrices}\label{section:Earle}
In \cite{Earle:2014-1}, Earle considered the following special family of \CS matrices
\[X_{c,d,c+2}=\begin{bmatrix}0&a&b\\0&c&d\\1&0&2\end{bmatrix},\]
and showed that some of them are Gompf equivalent to $A_0$.
\begin{theorem}[{\cite[Theorem~3.1]{Earle:2014-1}}]\label{theorem:Earle}
 The \CS matrix $X_{c,d,c+2}$ is Gompf equivalent to $A_0$ if $0\leq c\leq 94$ and $a\neq 19, 37$, or if $1\leq d\leq 35$.
\end{theorem}
Using our method, we generalize Theorem~\ref{theorem:Earle} by removing the technical conditions on the entry~$a$, and weakening the condition on the entry~$d$ as follows:
\begin{theorem}\label{theorem:Earlegeneralversion}The \CS matrix $X_{c,d,c+2}$ is Gompf equivalent to $A_0$ if $0\leq c\leq 94$ or if $1\leq d\leq 134$.
\end{theorem}
\begin{proof}By Theorem~\ref{theorem:B}, $X_{c,d,c+2}$ is Gompf equivalent to $A_0$ if $1\leq d\leq 134$. It suffices to prove for the cases that $a=19$ or $37$ and $0\leq c\leq 94$. By Proposition~\ref{csm2}, $f_{c+2}(c)\equiv 0\Mod{d}$ since $X_{c,d,c+2}$ is a \CS matrix. The following tuples $(c,d,c+2)$ in $\mathcal{CS}$ give the list of \CS matrices $X_{c,d,c+2}$ satisfying $a=19$ and $0\leq c\leq 94$:
\begin{align*}&(8,3,10), (12,7,14), (27,37,29), (31,49,33), (46,109,48), (50,129,52), (65,219,67),\\
& (69,247,71), (84,367,86), (88,403,90).
\end{align*}
The tuples in the first row correspond to \CS matrices with trace $\leq 69$, and hence Gompf equivalent to $A_0$ by Theorem~\ref{theorem:B}. We give Gompf equivalences from the tuples in the second row as we did in the proof of Theorem~\ref{theorem:B} to the tuples that are known to Gompf equivalent to $(1,1,2)$ using the MAGMA code for Algorithm~\ref{algorithm:2} given in Section~\ref{subsection:MAGMAcodes} as follows:\begin{itemize}
\item $(69,247,71)\sim_S (83,103,71)$.
\item $(84,367,86)\sim_S (102,127,86)$.
\item $(88,403,90)\sim_S (107,133,90)$.
\end{itemize}

Similarly, the following tuples $(c,d,c+2)$ in $\mathcal{CS}$ give the list of \CS matrices $X_{c,d,c+2}$ satisfying $a=37$ and $0\leq c\leq 94$:
\[(11,3,13), (27,19,29), (48,61,50), (64,109,66), (85,193,87).\]
As we did before, we give a Gompf equivalence from $(85,193,87)$ to a tuple that is known to Gompf equivalent to $(1,1,2)$ as follows:
\[(85,193,87)\sim_S (198,283,87)\sim_G(198,283,-196)\sim_S(155,229,-196)\sim_G(155,229,33).\]
This completes the proof.
\end{proof}

\bibliographystyle{amsalpha}
\renewcommand{\MR}[1]{}
\bibliography{research}
\end{document}